\documentclass[10pt]{amsart}
\usepackage{amsmath}
\usepackage{amsxtra}
\usepackage{amscd}
\usepackage{amsthm}
\usepackage{amsfonts}
\usepackage{amssymb}
\usepackage{eucal}
\usepackage[matrix,arrow,curve]{xy}

\newtheorem{cor}[subsection]{Corollary}

\newtheorem{lem}[subsection]{Lemma}
\newtheorem{prop}[subsection]{Proposition}

\newtheorem{conj}[subsection]{Conjecture}
\newtheorem{thm}[subsection]{Theorem}

\newtheorem{rem}[subsection]{Remark}

\theoremstyle{definition}

\theoremstyle{remark}

\newcommand{\nc}{\newcommand}
\nc{\renc}{\renewcommand} \nc{\ssec}{\subsection}
\nc{\sssec}{\subsubsection} \nc{\on}{\operatorname}

\nc\ol{\overline} \nc\ul{\underline} \nc\wt{\widetilde}
\nc\tboxtimes{\wt{\boxtimes}} \nc{\alp}{\alpha}

\nc{\ZZ}{{\mathbb Z}} \nc{\NN}{{\mathbb N}} \nc{\CC}{{\mathbb C}}
\nc{\OO}{{\mathbb O}} \renc{\SS}{{\mathbb S}} \nc{\DD}{{\mathbb
D}}

\nc{\Fq}{{\mathbb F}_q} \nc{\Fqb}{\ol{{\mathbb F}_q}}
\nc{\Ql}{\ol{{\mathbb Q}_\ell}} \nc{\id}{\text{id}} \nc\X{\mathcal
X}

\nc{\Hom}{\on{Hom}} \nc{\Lie}{\on{Lie}} \nc{\Loc}{\on{Loc}}
\nc{\Pic}{\on{Pic}} \nc{\Bun}{\on{Bun}} \nc{\IC}{\on{IC}}
\nc{\Aut}{\on{Aut}} \nc{\rk}{\on{rk}} \nc{\Sh}{\on{Sh}}
\nc{\Perv}{\on{Perv}} \nc{\pos}{{\on{pos}}} \nc{\Conv}{\on{Conv}}
\nc{\Sph}{\on{Sph}} \nc{\Sym}{\on{Sym}}
\nc{\BunBb}{\overline{\Bun}_B} \nc{\Buno}{\overset{o}{\Bun}}
\nc{\BunPb}{{\overline{\Bun}_P}}
\nc{\BunBM}{\overline{\Bun}_{B(M)}}
\nc{\BunPbw}{{\widetilde{\Bun}_P}}
\nc{\BunBP}{\widetilde{\Bun}_{B,P}} \nc{\GUb}{\overline{G/U}}
\nc{\GUPb}{\overline{G/U(P)}}
\nc{\iso}{{\stackrel{\sim}{\longrightarrow}}}

\nc{\Hhom}{\underline{\on{Hom}}} \nc\syminfty{\on{Sym}^{\infty}}
\nc\lal{\ol{\lambda}} \nc\xl{\ol{x}} \nc\thl{\ol{\theta}}
\nc\nul{\ol{\nu}} \nc\mul{\ol{\mu}} \nc\Sum\Sigma
\nc{\oX}{\overset{o}{X}{}}

\nc{\M}{{\mathcal M}} \nc{\N}{{\mathcal N}} \nc{\F}{{\mathcal F}}
\nc{\D}{{\mathcal D}} \nc{\Q}{{\mathcal Q}} \nc{\Y}{{\mathcal Y}}
\nc{\G}{{\mathcal G}} \nc{\E}{{\mathcal E}} \nc{\CalC}{{\mathcal
C}}
\nc\Dh{\widehat{\D}}

\nc{\C}{{\mathcal C}} \nc{\K}{{\mathcal K}}
\renewcommand{\H}{{\mathcal H}}

\nc{\T}{{\mathcal T}} \nc{\V}{{\mathcal V}} \renc{\P}{{\mathcal
P}} \nc{\A}{{\mathcal A}} \nc{\B}{{\mathcal B}} \nc{\U}{{\mathcal
U}}

\nc{\Gr}{\on{Gr}}

\nc{\frn}{{\check{\mathfrak u}(P)}}
\nc\f{{\mathfrak f}}

\nc{\q}{{\mathfrak q}} \nc{\p}{{\mathfrak p}} \nc{\s}{{\mathfrak
s}} \nc\w{\text{w}}

\nc\mathi\iota \nc\Spec{\on{Spec}} \nc\Mod{\on{Mod}}
\nc{\tw}{\widetilde{\mathfrak t}} \nc{\pw}{\widetilde{\mathfrak
p}} \nc{\qw}{\widetilde{\mathfrak q}} \nc{\jw}{\widetilde j}

\nc{\grb}{\overline{\Gr}} \nc{\I}{\mathcal I}

\nc{\lambdach}{{\check\lambda}} \nc{\Lambdach}{{\check\Lambda}{}}
\nc{\much}{{\check\mu}} \nc{\omegach}{{\check\omega}}
\nc{\nuch}{{\check\nu}} \nc{\etach}{{\check\eta}}
\nc{\alphach}{{\check\alpha}} \nc{\betach}{{\check\beta}}
\nc{\rhoch}{{\check\rho}} \nc{\ch}{{\check h}}

\nc{\Hb}{\overline{\H}}

\nc{\BA}{{\mathbb{A}}} \nc{\BC}{{\mathbb{C}}} \nc{\BQ}{{\mathbb{Q}}}
\nc{\BM}{{\mathbb{M}}} \nc{\BN}{{\mathbb{N}}}
\nc{\BP}{{\mathbb{P}}} \nc{\BR}{{\mathbb{R}}}
\nc{\BZ}{{\mathbb{Z}}} \nc{\BS}{{\mathbb{S}}}

\nc{\CA}{{\mathcal{A}}} \nc{\CB}{{\mathcal{B}}}
\nc{\CE}{{\mathcal{E}}} \nc{\CF}{{\mathcal{F}}}
\nc{\CG}{{\mathcal{G}}} \nc{\CH}{{\mathcal{H}}}
\nc{\CI}{{\mathcal{I}}} \nc{\CL}{{\mathcal{L}}}
\nc{\CM}{{\mathcal{M}}} \nc{\CN}{{\mathcal{N}}}
\nc{\CO}{{\mathcal{O}}} \nc{\CP}{{\mathcal{P}}}
\nc{\CQ}{{\mathcal{Q}}} \nc{\CR}{{\mathcal{R}}}
\nc{\CS}{{\mathcal{S}}} \nc{\CT}{{\mathcal{T}}}
\nc{\CU}{{\mathcal{U}}} \nc{\CV}{{\mathcal{V}}}  \nc{\CY}{{\mathcal Y}}
\nc{\CW}{{\mathcal{W}}} \nc{\CZ}{{\mathcal{Z}}}

\nc{\cM}{{\check{\mathcal M}}{}} \nc{\csM}{{\check{\mathcal A}}{}}
\nc{\oM}{{\overset{\circ}{\mathcal M}}{}}
\nc{\obM}{{\overset{\circ}{\mathbf M}}{}}
\nc{\oCA}{{\overset{\circ}{\mathcal A}}{}}
\nc{\obA}{{\overset{\circ}{\mathbf A}}{}}
\nc{\ooM}{{\overset{\circ}{M}}{}}
\nc{\osM}{{\overset{\circ}{\mathsf M}}{}}
\nc{\vM}{{\overset{\bullet}{\mathcal M}}{}}
\nc{\nM}{{\underset{\bullet}{\mathcal M}}{}}
\nc{\oD}{{\overset{\circ}{\mathcal D}}{}}
\nc{\obD}{{\overset{\circ}{\mathbf D}}{}}
\nc{\oA}{{\overset{\circ}{\mathbb A}}{}}
\nc{\op}{{\overset{\bullet}{\mathbf p}}{}}
\nc{\cp}{{\overset{\circ}{\mathbf p}}{}}
\nc{\oU}{{\overset{\bullet}{\mathcal U}}{}}
\nc{\oZ}{{\overset{\circ}{\mathcal Z}}{}}
\nc{\ofZ}{{\overset{\circ}{\mathfrak Z}}{}}

\nc{\ff}{{\mathfrak{f}}} \nc{\fv}{{\mathfrak{v}}}
\nc{\fa}{{\mathfrak{a}}} \nc{\fb}{{\mathfrak{b}}}
\nc{\fd}{{\mathfrak{d}}} \nc{\fe}{{\mathfrak{e}}}
\nc{\fg}{{\mathfrak{g}}} \nc{\fgl}{{\mathfrak{gl}}}
\nc{\fh}{{\mathfrak{h}}} \nc{\fri}{{\mathfrak{i}}}
\nc{\fj}{{\mathfrak{j}}} \nc{\fk}{{\mathfrak{k}}}
\nc{\fm}{{\mathfrak{m}}} \nc{\fn}{{\mathfrak{n}}}
\nc{\ft}{{\mathfrak{t}}} \nc{\fu}{{\mathfrak{u}}}
\nc{\fw}{{\mathfrak{w}}} \nc{\fz}{{\mathfrak{z}}}
\nc{\fp}{{\mathfrak{p}}} \nc{\frr}{{\mathfrak{r}}}
\nc{\fs}{{\mathfrak{s}}} \nc{\fsl}{{\mathfrak{sl}}}
\nc{\hsl}{{\widehat{\mathfrak{sl}}}}
\nc{\hgl}{{\widehat{\mathfrak{gl}}}}
\nc{\hg}{{\widehat{\mathfrak{g}}}}
\nc{\chg}{{\widehat{\mathfrak{g}}}{}^\vee}
\nc{\hn}{{\widehat{\mathfrak{n}}}}
\nc{\chn}{{\widehat{\mathfrak{n}}}{}^\vee}

\nc{\fA}{{\mathfrak{A}}} \nc{\fB}{{\mathfrak{B}}}
\nc{\fD}{{\mathfrak{D}}} \nc{\fE}{{\mathfrak{E}}}
\nc{\fF}{{\mathfrak{F}}} \nc{\fG}{{\mathfrak{G}}} \nc{\fH}{{\mathfrak{H}}}
\nc{\fI}{{\mathfrak{I}}} \nc{\fJ}{{\mathfrak{J}}}
\nc{\fK}{{\mathfrak{K}}} \nc{\fL}{{\mathfrak{L}}}
\nc{\fM}{{\mathfrak{M}}} \nc{\fN}{{\mathfrak{N}}}
\nc{\frP}{{\mathfrak{P}}} \nc{\fQ}{{\mathfrak{Q}}}
\nc{\fT}{{\mathfrak{T}}} \nc{\fU}{{\mathfrak{U}}}
\nc{\fV}{{\mathfrak{V}}} \nc{\fW}{{\mathfrak{W}}}
\nc{\fX}{{\mathfrak{X}}} \nc{\fY}{{\mathfrak{Y}}}
\nc{\fZ}{{\mathfrak{Z}}}

\nc{\ba}{{\mathbf{a}}}
\nc{\bb}{{\mathbf{b}}} \nc{\bc}{{\mathbf{c}}}
\nc{\be}{{\mathbf{e}}} \nc{\bj}{{\mathbf{j}}}
\nc{\bn}{{\mathbf{n}}} \nc{\bp}{{\mathbf{p}}}
\nc{\bq}{{\mathbf{q}}} \nc{\br}{{\mathbf{r}}} \nc{\bt}{{\mathbf{t}}}
\nc{\bfu}{{\mathbf{u}}} \nc{\bv}{{\mathbf{v}}}
\nc{\bx}{{\mathbf{x}}} \nc{\by}{{\mathbf{y}}}
\nc{\bw}{{\mathbf{w}}} \nc{\bA}{{\mathbf{A}}}
\nc{\bB}{{\mathbf{B}}} \nc{\bC}{{\mathbf{C}}}
\nc{\bD}{{\mathbf{D}}} \nc{\bF}{{\mathbf{F}}}
\nc{\bH}{{\mathbf{H}}} \nc{\bK}{{\mathbf{K}}}
\nc{\bM}{{\mathbf{M}}} \nc{\bN}{{\mathbf{N}}}
\nc{\bO}{{\mathbf{O}}} \nc{\bS}{{\mathbf{S}}} \nc{\bT}{{\mathbf{T}}}
\nc{\bV}{{\mathbf{V}}} \nc{\bW}{{\mathbf{W}}}
\nc{\bX}{{\mathbf{X}}}
\nc{\bY}{{\mathbf{Y}}} \nc{\bP}{{\mathbf{P}}}
\nc{\bZ}{{\mathbf{Z}}} \nc{\bh}{{\mathbf{h}}}

\nc{\sA}{{\mathsf{A}}} \nc{\sB}{{\mathsf{B}}}
\nc{\sC}{{\mathsf{C}}} \nc{\sD}{{\mathsf{D}}}
\nc{\sE}{{\mathsf{E}}} \nc{\sF}{{\mathsf{F}}}
\nc{\sK}{{\mathsf{K}}} \nc{\sL}{{\mathsf{L}}}
\nc{\sM}{{\mathsf{M}}} \nc{\sO}{{\mathsf{O}}}
\nc{\sQ}{{\mathsf{Q}}} \nc{\sP}{{\mathsf{P}}}
\nc{\sT}{{\mathsf{T}}} \nc{\sZ}{{\mathsf{Z}}}
\nc{\sV}{{\mathsf{V}}}
\nc{\sfp}{{\mathsf{p}}} \nc{\sr}{{\mathsf{r}}}
\nc{\st}{{\mathsf{t}}} \nc{\sfb}{{\mathsf{b}}}
\nc{\sfc}{{\mathsf{c}}} \nc{\sd}{{\mathsf{d}}}
\nc{\sz}{{\mathsf{z}}}

\nc{\BK}{{\bar{K}}}

\nc{\tA}{{\widetilde{\mathbf{A}}}}
\nc{\tB}{{\widetilde{\mathcal{B}}}}
\nc{\tg}{{\widetilde{\mathfrak{g}}}} \nc{\tG}{{\widetilde{G}}}
\nc{\TM}{{\widetilde{\mathbb{M}}}{}}
\nc{\tO}{{\widetilde{\mathsf{O}}}{}}
\nc{\tU}{{\widetilde{\mathfrak{U}}}{}} \nc{\TZ}{{\tilde{Z}}}
\nc{\tx}{{\tilde{x}}} \nc{\tbv}{{\tilde{\bv}}}
\nc{\tfP}{{\widetilde{\mathfrak{P}}}{}} \nc{\tz}{{\tilde{\zeta}}}
\nc{\tmu}{{\tilde{\mu}}}

\nc{\urho}{\underline{\rho}} \nc{\uB}{\underline{B}}
\nc{\uC}{{\underline{\mathbb{C}}}} \nc{\ui}{\underline{i}}
\nc{\uj}{\underline{j}} \nc{\ofP}{{\overline{\mathfrak{P}}}}
\nc{\oB}{{\overline{\mathcal{B}}}}
\nc{\og}{{\overline{\mathfrak{g}}}} \nc{\oI}{{\overline{I}}}

\nc{\eps}{\varepsilon} \nc{\hrho}{{\hat{\rho}}}
\nc{\blambda}{{\boldsymbol{\lambda}}}

\nc{\one}{{\mathbf{1}}} \nc{\two}{{\mathbf{t}}}

\nc{\Rep}{{\mathop{\operatorname{\rm Rep}}}}
\nc{\Tot}{{\mathop{\operatorname{\rm Tot}}}}
\nc{\Ker}{{\mathop{\operatorname{\rm Ker}}}}
\nc{\Hilb}{{\mathop{\operatorname{\rm Hilb}}}}
\nc{\End}{{\mathop{\operatorname{\rm End}}}}
\nc{\Ext}{{\mathop{\operatorname{\rm Ext}}}}
\nc{\CHom}{{\mathop{\operatorname{{\mathcal{H}}\it om}}}}
\nc{\GL}{{\mathop{\operatorname{\rm GL}}}}
\nc{\gr}{{\mathop{\operatorname{\rm gr}}}}
\nc{\Id}{{\mathop{\operatorname{\rm Id}}}}
\nc{\defi}{{\mathop{\operatorname{\rm def}}}}
\nc{\length}{{\mathop{\operatorname{\rm length}}}}
\nc{\supp}{{\mathop{\operatorname{\rm supp}}}}

\nc{\Cliff}{{\mathsf{Cliff}}}
\nc{\Fl}{{\mathsf{Fl}}} \nc{\Fib}{{\mathsf{Fib}}}
\nc{\Coh}{{\mathsf{Coh}}} \nc{\FCoh}{{\mathsf{FCoh}}}

\nc{\reg}{{\text{\rm reg}}}

\nc{\cplus}{{\mathbf{C}_+}} \nc{\cminus}{{\mathbf{C}_-}}
\nc{\cthree}{{\mathbf{C}_*}} \nc{\Qbar}{{\bar{Q}}}

\nc{\bOmega}{{\overline{\Omega}}}

\nc{\seq}[1]{\stackrel{#1}{\sim}}

\nc{\aff}{\operatorname{aff}}

\newcommand{\YO}{{\mathcal{Y}}}
\newcommand{\DO}{{\mathcal{D}}}
 \DeclareMathOperator{\diag}{diag}
 \DeclareMathOperator{\Tr}{Tr}
\DeclareMathOperator{\ad}{ad}

\begin{document}

\author{Michael Finkelberg and Leonid Rybnikov}
\title
{Quantization of Drinfeld Zastava in type $A$}

\dedicatory{To Borya Feigin on his 60th birthday}




\address{{\it Address}:\newline
M.F.: IMU, IITP, and
National Research University Higher School of Economics, \newline
Department of Mathematics,\newline
20 Myasnitskaya st,
Moscow 101000, Russia \newline
L.R.: IITP, and
National Research University Higher School of Economics, \newline
Department of Mathematics,\newline
20 Myasnitskaya st,
Moscow 101000, Russia}

\email{\newline fnklberg@gmail.com, leo.rybnikov@gmail.com}

\begin{abstract}
Drinfeld Zastava is a certain closure of the moduli space of maps from the
projective line to the Kashiwara flag scheme of the affine Lie algebra
$\hat{sl}_n$. We introduce an affine, reduced, irreducible, normal quiver
variety $Z$ which maps to the Zastava space bijectively at the level of
complex points. The natural Poisson structure on the Zastava space can be
described on $Z$ in terms of Hamiltonian reduction of a certain Poisson
subvariety of the dual space of a (nonsemisimple) Lie algebra. The quantum
Hamiltonian reduction of the corresponding quotient of its universal enveloping
algebra produces a quantization $Y$ of the coordinate ring of $Z$. The same
quantization was obtained in the finite (as opposed to the affine) case
generically in~\cite{o}. We prove that, for
generic values of quantization parameters, $Y$ is a quotient of the affine
Borel Yangian.
\end{abstract}
\maketitle

\section{Introduction}
\subsection{} The moduli space $\CP_{\ul{d}}^\circ$ of degree
$\ul{d}=(d_0,d_1,\ldots,d_{n-1})\in\BN^n$ based maps from the
projective line to the Kashiwara flag scheme of the affine Lie algebra
$\hsl_n$ admits two natural closures: an affine singular {\em Drinfeld
Zastava} space $Z^{\ul{d}}$, and a quasiprojective smooth {\em affine Laumon}
space $\CP_{\ul{d}}$ (see~\cite{fgk}). The advantage of $\CP_{\ul{d}}$ lies in
its smoothness (in fact, the natural proper morphism $\varpi:\ \CP_{\ul{d}}\to
Z^{\ul{d}}$ is a semismall resolution of singularities), while the advantage
of $Z^{\ul{d}}$ lies in the fact that it makes sense for other simple
and affine groups.

The affine Laumon space $\CP_{\ul{d}}$ is the moduli space of torsion free
parabolic sheaves on $\BP^1\times\BP^1$, and thus carries a natural Poisson
structure. This structure descends to the Drinfeld Zastava space $Z^{\ul{d}}$.
We have a natural problem to quantize this Poisson structure. The main goal of
our note is a solution of this problem. It was already solved
generically (on an open subvariety of $\CP_{\ul{d}}^\circ$) in the {\em finite},
i.e. $d_0=0$ (as opposed to the {\em affine}) case in~\cite{o}.

To this end we use a quiver construction of $\CP_{\ul{d}}$. This construction
follows from an observation by A.~Okounkov that $\CP_{\ul{d}}$ is a fixed
point set component of the cyclic group $\BZ/n\BZ$ acting on the moduli space
$\fM_{n,d}$ of torsion free sheaves on $\BP^1\times\BP^1$ framed at infinity.
The quiver in question (a {\em chainsaw} quiver) is similar to but different
from the $\tilde{A}_{n-1}$ quivers in Nakajima theory. In particular, the
corresponding quiver variety is {\em not} obtained by the Hamiltonian reduction
of a symplectic vector space. It is obtained by the Hamiltonian reduction
of a Poisson subvariety of the dual vector space of a (nonsemisimple) Lie
algebra $\fa_{\ul{d}}$ with its Lie-Kirillov-Kostant bracket.
The corresponding categorical
(as opposed to GIT) quotient $\fZ_{\ul{d}}$ is reduced, irreducible, normal,
and admits a morphism to the Zastava space $Z^{\ul{d}}$ which is bijective at
the level of $\BC$-points. We conjecture\footnote{This conjecture was
proved in~\cite{BF11}.} that this morphism is an isomorphism.

A historical comment is in order.
The quiver approach to Laumon moduli spaces goes back to
S.~A.~Str{\o}mme~\cite{s}; we have learnt of it from A.~Marian.
For a more recent construction of the monopole moduli space $\CP_{\ul{d}}^\circ$
in the finite (as opposed to the affine) case via Hamiltonian
reduction see~\cite{bp}. In fact, the authors of {\em loc. cit.} restrict
themselves to a single open coadjoint orbit in the Poisson subvariety of
the previous paragraph.

\subsection{} Now the ring of functions $\BC[\fZ_{\ul{d}}]$ admits a natural
quantization $\CY_{\ul{d}}$ as the quantum Hamiltonian reduction of a quotient
algebra of the universal enveloping algebra $U(\fa_{\ul{d}})$. The algebra
$\CY_{\ul{d}}$ admits a homomorphism from the Borel subalgebra $\YO$
of the Yangian of type $A_{n-1}$ in the case of finite Zastava space. We prove
that this homomorphism is surjective. In the affine situation, there is a 1-parametric deformation of $\fZ_{\ul{d}}$ analogous to the Calogero--Moser deformation of the Hilbert scheme. This leads to the the 1-parametric family of quantum Zastava spaces, $\CY_{\ul{d}}^\mu$. There is also an affine analog of $\YO$ depending on the complex parameter $\beta$ (we denote it $\widehat{\YO_\beta}$) in the same way as in \cite{g}. There is a homomorphism $\widehat{\YO_\beta}\to\YO_{\ul{d}}^\mu$ with $\beta=\mu+\sum\limits_{l=1}^nd_l$. We prove
that this homomorphism is surjective for $\mu\ne0$. Moreover, we write down certain elements in the kernel of this
homomorphism and conjecture that they generate the kernel (as a two-sided
ideal). These elements are similar to the generators of the kernel of the
surjective Brundan-Kleshchev homomorphism from their {\em shifted Yangian}
to a finite $W$-algebra of type $A$. In fact, it seems likely that
$\CY_{\ul{d}}$ as a filtered algebra is the limit of a sequence of finite
$W$-algebras of type $A$ equipped with the Kazhdan filtration.

Moreover, the similar quotients of the Borel Yangians for arbitrary simple
and affine Lie groups are likely to quantize the rings of functions on the
corresponding Drinfeld Zastava spaces.

\subsection{} Our motivation for quantization of Drinfeld Zastava came from
the following source. In~\cite{fr} we formulated a conjecture about
the {\em quantum connection} on equivariant quantum
cohomology of the finite Laumon spaces (it was proved recently by A.~Negut).
It identifies with the {\em Casimir}
connection, and its monodromy gives rise to an action of the pure braid group
on the equivariant cohomology of $\CP_{\ul{d}}$. According to the
Bridgeland-Bezrukavnikov-Okounkov philosophy, if we transfer this action to the
equivariant $K$-theory via Chern character, then it should come from an action
of the pure braid group on the equivariant derived category of coherent sheaves
on $\CP_{\ul{d}}$.

In the classical case of Nakajima quiver varieties, there are {\em chambers}
in the space of stability conditions for the GIT construction of quiver
varieties, and the derived coherent categories for the varieties in adjacent
chambers are related by Kawamata-type derived equivalences. These equivalences
generate the action of the pure braid group on the derived category of a single
quiver variety. Unfortunately, this approach fails in our situation (see
sections~\ref{vari}--\ref{smo}): although we do have chambers in the space
of stability conditions, the Laumon varieties in the adjacent chambers too
often become singular and just isomorphic (as opposed to birational).

Another approach was discovered by Bezrukavnikov-Mirkovi\'c in their works
on localization of $\fg$-modules in characteristic $p$. In our situation it
works as follows: if we replace the field $\BC$ of complex numbers by
an algebraic closure $\sK$ of a finite field of characteristic $p\gg0$,
then the quantized algebra $\CY_{\ul{d}}$ acquires a big center, isomorphic
to $\sK[\fZ^{(1)}_{\ul{d}}]$ (Frobenius twist of $\fZ_{\ul{d}}$). Thus
$\CY_{\ul{d}}$ may be viewed as global sections of a sheaf of noncommutative
algebras on $\fZ^{(1)}_{\ul{d}}$. A slight upgrade of our quantization
construction produces a sheaf $\CA_\chi$ of noncommutative algebras on
$\CP^{(1)}_{\ul{d}}$ for every stability condition $\chi$.
In sections~\ref{loc}--\ref{cohs} we formulate ``standard conjectures'' about
the sheaves of algebras $\CA_\chi$. We conjecture that they are all Morita
equivalent, and their global sections are isomorphic to $\CY_{\ul{d}}$.
Moreover, the functor of global sections from the category of $\CA_\chi$-modules
to the category of $\CY_{\ul{d}}$-modules is a derived equivalence for $\chi$
in certain chambers. Thus, for $\chi$ in such a chamber (e.g. $\chi=0$),
the composition of this derived equivalence with the above Morita equivalences
gives rise to an action of the pure braid group on $D^b(\CA_\chi$-mod).

Contrary to the Bezrukavnikov-Mirkovi\'c situation, in our case $\CA_\chi$
is {\em not} a sheaf of Azumaya algebras (e.g. in the simplest case $n=2,\
\ul{d}=(0,1)$, we have $\CP_{\ul{d}}\simeq\BA^2$, and $\CY_{\ul{d}}$ is the
universal enveloping algebra of the Borel subalgebra of $\fsl_2$).
However, in the formal neighbourhood of the central fiber of
$\varpi^{(1)}: \CP^{(1)}_{\ul{d}}\to\fZ^{(1)}_{\ul{d}}$,
the algebra $\CA_\chi$ possesses a
{\em splitting} module $\widehat M$. Tensoring with $\widehat M$ defines
a functor from the category of equivariant coherent sheaves on this formal
neighbourhood to the category of equivariant $\CA_\chi$-modules. We conjecture
that this functor is a full embedding, and the braid group action of the
previous paragraph preserves the essential image of this functor, thus giving
rise to the braid group action on the equivariant derived category of coherent
sheaves on the formal neighbourhood of the central fiber.

\subsection{Acknowledgments}
We are grateful to R.~Bezrukavnikov, A.~Braverman, B.~Feigin, V.~Ginzburg, A.~Molev
and V.~Vologodsky for useful discussions. During the key stage of the preparation
of this paper we have benefited from the hospitality and support of the
University of Sydney. Thanks are due to A.~Tsymbaliuk and J.~Kamnitzer
for the careful reading of the first
draft of this note and spotting several mistakes.

Both authors were partially supported by the RFBR grants 12-01-00944,
12-01-33101,
the National Research University Higher School of Economics' Academic Fund
award No.12-09-0062 and
the AG Laboratory HSE, RF government grant, ag. 11.G34.31.0023.
This study was carried out within the National Research University Higher School of Economics
Academic Fund Program in 2012-2013, research grant No. 11-01-0017.
This study comprises research findings from the ``Representation Theory
in Geometry and in Mathematical Physics" carried out within The
National Research University Higher School of Economics' Academic Fund Program
in 2012, grant No 12-05-0014. L.~R. was also partially
supported by the RFBR-CNRS grants 10-01-93111 and 11-01-93105, the RFBR grant 10-01-92104-JP-a,
and the National Research University Higher School of Economics' Academic Fund award No.10-01-0078.

\section{A quiver approach to Drinfeld and Laumon spaces}

\subsection{Parabolic sheaves}
\label{PS}
We recall the setup of Section~3 of~\cite{fgk}.
Let $\bC$ be a smooth projective
curve of genus zero. We fix a coordinate $z$ on $\bC$, and consider
the action of $\BC^*$ on $\bC$ such that $a(t)=a^{-1}\cdot t$. We have
$\bC^{\BC^*}=\{0_\bC,\infty_\bC\}$.
Let $\bX$ be another smooth
projective curve of genus zero. We fix a coordinate $y$ on $\bX$,
and consider the action of $\BC^*$ on $\bX$ such that
$c(x)=c^{-1}\cdot x$. We have $\bX^{\BC^*}=\{0_\bX,\infty_\bX\}$. Let
$\bS$ denote the product surface $\bC\times\bX$. Let $\bD_\infty$
denote the divisor $\bC\times\infty_\bX\cup\infty_\bC\times\bX$.
Let $\bD_0$ denote the divisor $\bC\times0_\bX$.

Given an $n$-tuple of nonnegative integers
$\ul{d}=(d_0,\ldots,d_{n-1})$, we say that a {\em parabolic sheaf}
$\CF_\bullet$ of degree $\ul{d}$ is an infinite flag of torsion
free coherent sheaves of rank $n$ on $\bS:\
\ldots\subset\CF_{-1}\subset\CF_0\subset\CF_1\subset\ldots$ such
that:

(a) $\CF_{k+n}=\CF_k(\bD_0)$ for any $k$;

(b) $ch_1(\CF_k)=k[\bD_0]$ for any $k$: the first Chern classes
are proportional to the fundamental class of $\bD_0$;

(c) $ch_2(\CF_k)=d_i$ for $i\equiv k\pmod{n}$;

(d) $\CF_0$ is locally free at $\bD_\infty$ and trivialized at
$\bD_\infty:\ \CF_0|_{\bD_\infty}=W\otimes\CO_{\bD_\infty}$;

(e) For $-n\leq k\leq0$ the sheaf $\CF_k$ is locally free at
$\bD_\infty$, and the quotient sheaves $\CF_k/\CF_{-n},\
\CF_0/\CF_k$ (both supported at $\bD_0=\bC\times0_\bX\subset\bS$)
are both locally free at the point $\infty_\bC\times0_\bX$;
moreover, the local sections of $\CF_k|_{\infty_\bC\times \bX}$
are those sections of $\CF_0|_{\infty_\bC\times
\bX}=W\otimes\CO_\bX$ which take value in $\langle
w_1,\ldots,w_{n+k}\rangle\subset W$ at $0_\bX\in \bX$.

\medskip

The fine moduli space
$\CP_{\ul{d}}$ of degree $\ul{d}$ parabolic sheaves exists and is
a smooth connected quasiprojective variety of dimension
$2d_0+\ldots+2d_{n-1}$.

\subsection{Parabolic sheaves as orbifold sheaves}
\label{realization}
We will now introduce a different realization of parabolic sheaves.
We first learned of this construction from A.~Okounkov,
though it is already present in the work of I.~Biswas~\cite{bi},
and goes back to M.~Narasimhan. Let
$\sigma:\bC\times \bX\rightarrow \bC\times \bX$ denote the map
$\sigma(z,y)=(z,y^n)$, and let $\Gamma=\BZ/n\BZ$. Then $\Gamma$ acts on
$\bC\times \bX$
by multiplying the coordinate on $\bX$ with the $n-$th roots of unity.
More precisely, we choose a generator $\gamma$ of $\Gamma$ which multiplies
$y$ by $\exp(\frac{2\pi i}{n})$.
We introduce a decreasing filtration $W=W^1=\langle w_1,\ldots,w_n\rangle
\supset W^2=\langle w_2,\ldots, w_n\rangle\supset\ldots\supset W^n=\langle w_n
\rangle$.

A parabolic sheaf $\CF_\bullet$ is completely determined by the flag of sheaves
$$
\CF_0(-\bD_0)\subset \CF_{-n+1}\subset...\subset \CF_0,
$$
satisfying conditions~\ref{PS}.(a--e). For $-n<k\leq0$ we consider a
subsheaf $\tilde\CF_k\subset\sigma^*\CF_k$ defined as follows. Away from
the line $\bC\times\infty_\bX$ the sheaf $\tilde\CF_k$ coincides with
$\sigma^*\CF_k$; and the local sections of $\tilde\CF_k|_{\bC\times\infty_\bX}$
are those sections of $\sigma^*\CF_k|_{\bC\times\infty_\bX}=
W\otimes\CO_{\bC\times\infty_\bX}$ which take value in $W^{k+n}$.

To $\CF_\bullet$ we can associate a single
$\Gamma$-equivariant torsion free sheaf $\tilde \CF$ on $\bC\times \bX$:
$$
\tilde \CF:=\tilde\CF_{-n+1}+
\tilde\CF_{-n+2}(\bC\times\infty_\bX-\bC\times0_\bX)+...
+\tilde\CF_0((n-1)(\bC\times\infty_\bX-\bC\times0_\bX)).
$$
Note that $\tilde\CF|_{\bC\times\infty_\bX}\equiv
W\otimes\CO_{\bC\times\infty_\bX}$, and $\tilde\CF|_{\infty_\bC\times\bX}$
is a trivial vector bundle, hence its trivialization on $\bC\times\infty_\bX$
canonically extends to a trivialization on $\bD_\infty$.

The sheaf $\tilde\CF$ will have to satisfy certain numeric and framing
conditions that
mimick conditions~\ref{PS}.(b--e). Conversely, any $\Gamma$-equivariant sheaf
$\tilde \CF$ that satisfies those numeric and framing conditions will
determine a unique parabolic sheaf. More precisely, for $d=d_0+\ldots+d_{n-1}$,
let $\fM_{n,d}$ be the Giesecker moduli space of torsion free sheaves on
$\bC\times\bX$ of rank $n$ and second Chern class $d$, trivialized on
$\bD_\infty$ (see~\cite{nak}, section~2). Then we have $\tilde\CF\in\fM_{n,d}$.
We consider the following action of $\Gamma$ on
$W:\ \gamma(w_l)=\exp(\frac{2\pi il}{n})w_l,\ l=1,\ldots,n$. The action of
$\Gamma$ on $\bC\times\bX$ together with its action on the trivialization
at $\bD_\infty$ (via the action on $W$)
gives rise to the action of $\Gamma$ on $\fM_{n,d}$. We have
$\tilde\CF\in\fM_{n,d}^\Gamma$. Thus we have constructed an embedding
$\CP_{\ul{d}}\hookrightarrow\fM_{n,d}^\Gamma,\ \CF_\bullet\mapsto\tilde\CF$.
The fixed point set $\fM_{n,d}^\Gamma$ has many connected components numbered
by decompositions $d=d_0+d_1+\ldots+d_{n-1}$, and the embedding
$\CP_{\ul{d}}\hookrightarrow\fM_{n,d}^\Gamma$ is an isomorphism onto the
connected component $\fM_{n,\ul{d}}^\Gamma$.

The inverse isomorphism takes a $\Gamma$-equivariant torsion free sheaf
$\tilde\CF$ to the flag $\CF_0(-\bD_0)\subset\CF_{-n+1}\subset\ldots\subset
\CF_0$ where for $-n<k\leq0$ we set
$\CF_k:=\sigma_*\left(\tilde\CF\otimes\CO_\bS(k\bD_0)\right)^\Gamma$.

\subsection{A quiver description of Laumon space}
\label{queer}
According to section~2 of~\cite{nak}, $\fM_{n,d}$ admits the following GIT
description. We set $V=\BC^d$, and we consider $M=\End(V)\oplus\End(V)
\oplus\Hom(W,V)\oplus\Hom(V,W)$. A typical quadruple in $M$ will be denoted
by $(A,B,p,q)$. We set $L\supset\mu^{-1}(0):=\{(A,B,p,q):\ AB-BA+pq=0\}$.
We define $\mu^{-1}(0)^s$ as the open subset of stable quadruples, i.e. those
which do not admit proper subspaces $V'\subset V$ stable under $A,B$ and
containing $p(W)$. The group $GL(V)$ acts naturally on $M$ preserving
$\mu^{-1}(0)$; its action on
$\mu^{-1}(0)^s$ is free, and $\fM_{n,d}$ is the GIT quotient
$\mu^{-1}(0)^s/GL(V)$.

In terms of this quiver realization, the action of $\Gamma$ is described
as follows: $\gamma(A,B,p,q)=(A,\exp(\frac{2\pi i}{n})B,
\exp(\frac{2\pi i}{n})p,q)$. Recall that the action of $\Gamma$ on $W$ was
desribed in~\ref{realization}: for $l=1,\ldots,n,\ W_l=\langle w_l\rangle$
is the isotypic component corresponding to the character
$\chi_l(\gamma)=\exp(\frac{2\pi il}{n})$. Hence the connected component
of the fixed point set $\CP_{\ul{d}}\simeq\fM_{n,\ul{d}}^\Gamma$ admits
the following quiver description.

We choose an action of $\Gamma$ on $V$
such that the $\chi_l$-isotypic component $V_l$ has dimension $d_l\
(l\in\BZ/n\BZ)$. Then $M^\Gamma_{\ul{d}}=\{(A_l,B_l,p_l,q_l)_{l\in\BZ/n\BZ}\}=$
$$\bigoplus_{l\in\BZ/n\BZ}\End(V_l)\oplus
\bigoplus_{l\in\BZ/n\BZ}\Hom(V_l,V_{l+1})\oplus
\bigoplus_{l\in\BZ/n\BZ}\Hom(W_{l-1},V_l)\oplus
\bigoplus_{l\in\BZ/n\BZ}\Hom(V_l,W_l):$$

$$\xymatrix{
\ldots \ar[r]^{B_{-3}}
& V_{-2} \ar@(ur,ul)[]_{A_{-2}} \ar[r]^{B_{-2}} \ar[d]_{q_{-2}}
& V_{-1} \ar@(ur,ul)[]_{A_{-1}} \ar[r]^{B_{-1}} \ar[d]_{q_{-1}}
& V_0 \ar@(ur,ul)[]_{A_0} \ar[r]^{B_0} \ar[d]_{q_0}
& V_1 \ar@(ur,ul)[]_{A_1} \ar[r]^{B_1} \ar[d]_{q_1}
& V_2 \ar@(ur,ul)[]_{A_2} \ar[r]^{B_2} \ar[d]_{q_2} &\ldots\\
\ldots \ar[ur]^{p_{-2}} & W_{-2} \ar[ur]^{p_{-1}} & W_{-1} \ar[ur]^{p_0}
& W_0 \ar[ur]^{p_1} & W_1 \ar[ur]^{p_2} & W_2 \ar[ur]^{p_3} &\ldots
}$$
(the {\em chainsaw quiver}).

Furthermore, $\mu^{-1}(0)^\Gamma_{\ul{d}}=\{(A_l,B_l,p_l,q_l)_{l\in\BZ/n\BZ}:\
A_{l+1}B_l-B_lA_l+p_{l+1}q_l=0\ \forall l\}$.
Moreover, $\mu^{-1}(0)^{s,\Gamma}_{\ul{d}}=
\{(A_l,B_l,p_l,q_l)_{l\in\BZ/n\BZ}\in
\mu^{-1}(0)^\Gamma_{\ul{d}}:$ there is no proper $\BZ/n\BZ$-graded subspace
$V'_\bullet\subset V_\bullet$ stable under $A_\bullet,B_\bullet$ and
containing $p(W_\bullet)\}$.

Finally, the group $\prod_{l\in\BZ/n\BZ}GL(V_l)$ acts naturally on
$M^\Gamma_{\ul{d}}$
preserving $\mu^{-1}(0)^\Gamma_{\ul{d}}$; its action on
$\mu^{-1}(0)^{s,\Gamma}_{\ul{d}}$ is free, and $\fM_{n,\ul{d}}=
\mu^{-1}(0)^{s,\Gamma}_{\ul{d}}/\prod_{l\in\BZ/n\BZ}GL(V_l)$.

\begin{rem}
\label{F0}
{\em If a point $\CF_\bullet\in\CP_{\ul{d}}\simeq\fM_{n,\ul{d}}$ has
a representative $(A_l,B_l,p_l,q_l)_{l\in\BZ/n\BZ}$, then $\CF_0\in\CM_{n,d_0}$
has a representative $(A',B',p',q')$ defined as follows. First of all,
$W'=W_0\oplus W_1\oplus\ldots\oplus W_{n-1},\ V'=V_0$. Now $A'=A_0,\
B'=B_{n-1}B_{n-2}\ldots B_1B_0,\ p'=\oplus_{0\leq
l\leq n-1}B_{n-1}B_{n-2}\ldots B_{l}p_l,\ q'=\oplus_{0\leq l\leq n-1}
q_lB_{l-1}\ldots B_1B_0$.}
\end{rem}

\begin{rem}
\label{neg}
{\em A.~Negut has introduced in~\cite{n} the moduli spaces
$\CM'_{\ul{d}}$ closely related to Laumon moduli spaces. Namely,
$\CM'_{\ul{d}}$ is defined as the moduli space of flags of locally free
sheaves $0\subset\CF_1\subset\ldots\subset\CF_{n-1}\subset\CF_n\subset
W\otimes\CO_\bC$ such that $\on{rk}\CF_k=k,\ k=1,\ldots,n;\
\deg\CF_k=-d_k$, and at $\infty_\bC$ our flag consists of vector subbundles,
and takes value $\langle w_1\rangle\subset\langle w_1,w_2\rangle\subset
\ldots\langle w_1,\ldots,w_{n-1}\rangle\subset W$.

Let us consider the following {\em handsaw quiver} $Q'$

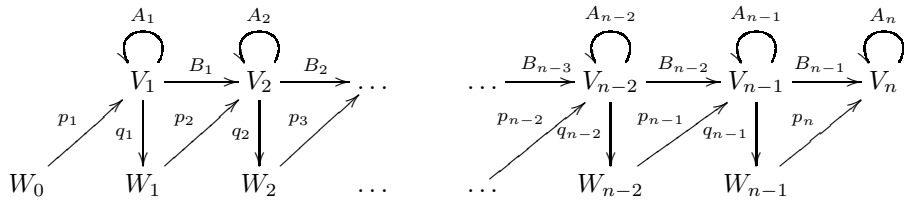
$$\xymatrix{
& V_1 \ar@(ur,ul)[]_{A_1} \ar[r]^{B_1} \ar[d]_{q_1}
& V_2 \ar@(ur,ul)[]_{A_2} \ar[r]^{B_2} \ar[d]_{q_2} &\ldots
&\ldots \ar[r]^{B_{n-3}}
& V_{n-2} \ar@(ur,ul)[]_{A_{n-2}} \ar[r]^{B_{n-2}} \ar[d]_{q_{n-2}}
& V_{n-1} \ar@(ur,ul)[]_{A_{n-1}} \ar[r]^{B_{n-1}} \ar[d]_{q_{n-1}}
& V_n \ar@(ur,ul)[]_{A_n}\\
W_0 \ar[ur]^{p_1} & W_1 \ar[ur]^{p_2} & W_2 \ar[ur]^{p_3} &\ldots
&\ldots \ar[ur]^{p_{n-2}} & W_{n-2} \ar[ur]^{p_{n-1}} & W_{n-1} \ar[ur]^{p_n}
}$$
with relations $A_{k+1}B_k-B_kA_k+p_{k+1}q_k=0,\ k=1,\ldots,n-1$.
Let $\sM'_{\ul{d}}$ stand for the moduli scheme of representations of $Q'$
(quiver with relations) such that $\dim W_0=\ldots=\dim W_{n-1}=1,\
\dim V_k=d_k,\ k=1,\ldots,n$. Let ${\sM_{\ul{d}}^s}'$ stand for the open
subscheme of stable representations of $Q'$ formed by all the quadruples
$(A_\bullet,B_\bullet,p_\bullet,q_\bullet)$ such that there is no proper
graded subspace $V'_\bullet\subset V_\bullet$ stable under $A_\bullet,
B_\bullet$ and containing $p_\bullet(W_\bullet)$. Let $G_{\ul{d}}$ stand
for the group $\prod_{k=1}^n GL(V_k)$ acting on $\sM_{\ul{d}}'$ naturally.
Then the action of $G_{\ul{d}}$ on ${\sM_{\ul{d}}^s}'$ is free,
and the argument
of Sections~\ref{realization} and ~\ref{queer} proves that the
quotient ${\sM_{\ul{d}}^s}'/G_{\ul{d}}$ is isomorphic to $\CM'_{\ul{d}}$.}
\end{rem}

\subsection{A quiver approach to Drinfeld Zastava}
\label{Queer}
We define $\fZ_{\ul{d}}$ as the categorical quotient
$\mu^{-1}(0)^\Gamma_{\ul{d}}//\prod_{l\in\BZ/n\BZ}GL(V_l)$, that is the
spectrum of the ring of $\prod_{l\in\BZ/n\BZ}GL(V_l)$-invariants in
$\BC[\mu^{-1}(0)^\Gamma_{\ul{d}}]$.

Let $\chi=\chi_{-1,\ldots,-1}$ stand for the character $(g_1,\ldots,g_n)\mapsto
\det(g_1)\ldots\det(g_n):\ \prod_{l\in\BZ/n\BZ}GL(V_l)\to\BC^*$.
Let us denote $\prod_{l\in\BZ/n\BZ}GL(V_l)$ by $G_{\ul{d}}$ for short.

Let $\BC[\mu^{-1}(0)^\Gamma_{\ul{d}}]^{G_{\ul{d}},\chi^r}$ stand for the
$\chi^r$-isotypical component of $\BC[\mu^{-1}(0)^\Gamma_{\ul{d}}]$ under the
action of $G_{\ul{d}}$. Then $\fM_{n,\ul{d}}=
\mu^{-1}(0)^{s,\Gamma}_{\ul{d}}/G_{\ul{d}}=
\on{Proj}\left(\bigoplus_{r\geq0}
\BC[\mu^{-1}(0)^\Gamma_{\ul{d}}]^{G_{\ul{d}},\chi^r}\right)$.
We have a projective morphism $\pi:\ \fM_{n,\ul{d}}\to\fZ_{\ul{d}}$.

Let $Z^{\ul{d}}$ stand for the Drinfeld Zastava space defined (under the
name of $\fM^\alpha$) in section~4 of~\cite{fgk} and (for an arbitrary
almost simple simply connected group $G$ in place of $\on{SL}(n)$ here)
in~\cite{bfg}. Let $\varpi:\ \CP_{\ul{d}}\to Z^{\ul{d}}$ be the morphism
(semismall resolution of singularities) introduced in section~5 of~\cite{fgk}.
Our next goal is to prove the following

\begin{thm}
\label{bij}
a) $\fZ_{\ul{d}}$ is a reduced irreducible normal scheme.

b) The morphism $\varpi:\ \CP_{\ul{d}}\to Z^{\ul{d}}$ factors as
$\CP_{\ul{d}}\stackrel{\pi}{\to}\fZ_{\ul{d}}\stackrel{\eta}{\to}Z^{\ul{d}}$,
and $\eta$ induces a bijection between the sets of $\BC$-points.
\end{thm}

The proof occupies the rest of this section.

\subsection{Examples}
\label{ex}
We consider three basic examples of Zastava spaces for the groups
$\on{SL}(2),\on{SL}(3),\widehat{\on{SL}}(2)$.

\subsubsection{$\on{SL}(2)$}
\label{sl2}
We take $n\geq2,\ d_2=d_3=\ldots=d_{n-1}=d_0=0,\ d_1=d$.
We have $V_1=V=\BC^d,\ A_1=A\in\End(V),\ B_1=0,\ p_1=p\in V,\ q_1=q\in V^*,\
G_{\ul{d}}=GL(V)$.
Thus $\mu^{-1}(0)=\End(V)\oplus V\oplus V^*$, and $\fZ_{\ul{d}}=
(\End(V)\oplus V\oplus V^*)//GL(V)$. By the classical Invariant Theory,
the ring of $GL(V)$-invariant functions on $\End(V)\oplus V\oplus V^*$ is
freely generated by the functions $a_1,\ldots,a_d,b_0,\ldots,b_{d-1}$ where
$a_m:=\on{Tr}(A^m)$, and $b_m:=q\circ A^m\circ p$.
Hence $\fZ_{\ul{d}}\simeq\BA^{2d}$.

\subsubsection{$\on{SL}(3)$}
\label{sl3}
We take $n\geq3,\ d_3=d_4=\ldots=d_{n-1}=d_0=0,\ d_1=d_2=1$.
We have $V_1=\BC=V_2$, and hence all our linear operators act between
one-dimensional vector spaces, and can be written just as numbers.
We have nonzero numbers $A_1,A_2,B_1,p_1,p_2,q_1,q_2$, and
$\mu^{-1}(0)$ is given by the single equation $B_1(A_2-A_1)+p_2q_1=0$.
The group $G_{\ul{d}}$ is just $\BC^*\times\BC^*$ with coordinates
$c_1,c_2$. It acts on $\mu^{-1}(0)$ as follows:
$(c_1,c_2)\cdot(A_1,A_2,B_1,p_1,p_2,q_1,q_2)=
(A_1,A_2,c_1c_2^{-1}B_1,c_1^{-1}p_1,c_2^{-1}p_2,c_1q_1,c_2q_2)$.
The ring of $\BC^*\times\BC^*$-invariant functions on $\mu^{-1}(0)$
is generated by the functions $b_{1,0}:=q_1p_1,\ b_{2,0}:=q_2p_2,\
r:=q_2B_1p_1,\ A_1,\ A_2$ with a single relation $b_{1,0}b_{2,0}+r(A_2-A_1)=0$.
Thus, $\fZ_{\ul{d}}$ is the product of the conifold with the affine line.

\subsubsection{$\widehat{\on{SL}}(2)$}
\label{hatsl2}
We take $n=2,\ d_0=d_1=1$. We have $V_1=\BC=V_2$, and hence all
our linear operators act between
one-dimensional vector spaces, and can be written just as numbers.
We have nonzero numbers $A_1,A_0,B_1,B_0,p_1,p_0,q_1,q_0$, and
$\mu^{-1}(0)$ is cut out by two equations
$B_1(A_0-A_1)+p_0q_1=0=B_0(A_1-A_0)+p_1q_0$.
The group $G_{\ul{d}}$ is just $\BC^*\times\BC^*$ with coordinates
$c_1,c_0$. It acts on $\mu^{-1}(0)$ as follows:
$(c_1,c_0)\cdot(A_1,A_0,B_1,B_0,p_1,p_0,q_1,q_0)=
(A_1,A_0,c_1c_0^{-1}B_1,c_0c_1^{-1}B_0,c_1^{-1}p_1,c_0^{-1}p_0,c_1q_1,c_0q_0)$.
The ring of $\BC^*\times\BC^*$-invariant functions on $\mu^{-1}(0)$
is generated by the functions $b_{1,0}:=q_1p_1,\ b_{0,0}:=q_0p_0,\
s:=B_0B_1,\ A_1,\ A_0$ with a single relation $b_{1,0}b_{0,0}-s(A_0-A_1)^2=0$.

\subsection{Stratification of $\fZ_{\ul{d}}$}
\label{strat}
Applying the famous Crawley-Boevey's trick we may identify all the
one-dimensional spaces $W_l$, and denote the resulting line by $W_\infty$.
Thus, $W_\infty$ becomes the source of all $p_l$, and the target of all $q_l,\
l\in\BZ/n\BZ$:
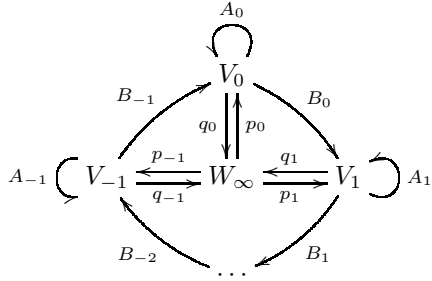
$$\xymatrix{
& V_0 \ar@(ur,ul)[]_{A_0} \ar@/^/[rd]^{B_0} \ar@<-.5ex>[d]_{q_0} &\\
V_{-1} \ar@(ul,dl)[]_{A_{-1}} \ar@/^/[ru]^{B_{-1}} \ar@<-.5ex>[r]_{q_{-1}}
& W_\infty \ar@<-.5ex>[l]_{p_{-1}} \ar@<-.5ex>[u]_{p_0} \ar@<-.5ex>[r]_{p_1}
& V_1 \ar@(dr,ur)[]_{A_1} \ar@/^/[ld]^{B_1} \ar@<-.5ex>[l]_{q_1}\\
& \cdots \ar@/^/[lu]^{B_{-2}} &
}$$
The $\BC$-points of $\fZ_{\ul{d}}$ classify the semisimple
representations of the resulting {\em Ferris wheel quiver} with relations
$\mu=0$, to be denoted
by $Q$. More precisely, the $\BC$-points of $\fZ_{\ul{d}}$ classify the
semisimple $Q$-modules of dimension $\ul{\dim}=
(\dim(V_l)_{l\in\BZ/n\BZ},\dim(W_\infty)):\ \dim(W_\infty)=1,\ \dim(V_l)=d_l$.

We start with the classification of simple $Q$-modules of dimension
smaller than or equal to $\ul{\dim}$.
First suppose $\dim(W_\infty)=0$. Then an irreducible module is either
$L_l(x)$ for some $l\in\BZ/n\BZ,\ x\in\BC$, or $L(x,y)$ for some
$x\in\BC,\ y\in\BC^*$. Here $L_l(x)$ denotes the $Q$-module with
$V_k=0$ for $k\ne l$, and $V_l=\BC,\ A_l=x$. Furthermore,
$L(x,y)$ denotes the $Q$-module with $V_l=\BC,\ A_l=x\ \forall l\in\BZ/n\BZ,\
\prod_{l\in\BZ/n\BZ}B_l=y$.

Now suppose $\dim(W_\infty)=1$. Then the irreducibility condition is equivalent
to the conjunction of stability condition of~\ref{queer} and of costability:
there is no proper $\BZ/n\BZ$-graded subspace $V'_\bullet\subset V_\bullet$
stable under $A_\bullet,B_\bullet$ and contained in $\on{Ker}(q_\bullet)$.
We will denote the open subset of stable and costable $Q$-modules
of dimension $(1,\ul{d}')\leq(1,\ul{d})$ by
$\mu^{-1}(0)^{sc,\Gamma}_{\ul{d}'}$. According to Chapter~2 of~\cite{nak},
the open subset
$\fZ_{\ul{d}'}\supset\mu^{-1}(0)^{sc,\Gamma}_{\ul{d}'}/G_{\ul{d}'}\subset
\mu^{-1}(0)^{s,\Gamma}_{\ul{d}'}/G_{\ul{d}'}=\CP_{\ul{d}'}$ coincides with
the moduli space of {\em locally free} parabolic sheaves, to be denoted by
$\CP^\circ_{\ul{d}'}$. Thus, the isomorphism classes of irreducible
$Q$-modules of dimension $(1,\ul{d}')$ are parametrized by
$\CP^\circ_{\ul{d}'}$.

We conclude that the set of $\BC$-points of $\fZ_{\ul{d}}$ is a disjoint
union of the following strata. We fix an $n$-tuple $\ul{d}'\leq\ul{d}$,
a collection of positive integers $m_1,\ldots,m_r$, and also collections
of positive integers $(m_{l1},\ldots,m_{l,r_l})_{l\in\BZ/n\BZ}$
such that for any $l$ we have
$d_l=d'_l+\sum_{i=1}^r m_i+\sum_{j=1}^{r_l}m_{lj}$.
Then the corresponding stratum is formed by the isomorphism classes of
semisimple $Q$-modules of type $R\oplus \bigoplus_{i=1}^r
L(x_i,y_i)^{\oplus m_i}\oplus
\bigoplus_{l\in\BZ/n\BZ}\bigoplus_{j=1}^{r_l}L_l(x_j)^{\oplus m_{lj}}$ where
$R\in\CP^\circ_{\ul{d}'}$, and all the pairs $(x_i,y_i)_{i=1,\ldots,m_r}$
are distinct, and for any $l$ all the points $x_j,\ j=1,\ldots,m_{l,r_l}$,
are distinct.

\subsection{Dimension of $\mu^{-1}(0)^\Gamma_{\ul{d}}$}
\label{estim}
We consider the configuration space of $\BZ/n\BZ$-colored points
$\BA^{\ul{d}}:=(\bC-\infty_\bC)^{(d_0)}\times\ldots\times
(\bC-\infty_\bC)^{(d_{n-1})}$. We denote $\mu^{-1}(0)^\Gamma_{\ul{d}}$ by
$\sM_{\ul{d}}$ for short. We have a morphism $\Upsilon:\ \sM_{\ul{d}}\to
\BA^{\ul{d}}$ sending a quadruple $(A_\bullet,B_\bullet,p_\bullet,q_\bullet)$
to $(\on{Spec}A_0,\ldots,\on{Spec}A_{n-1})$.

\begin{prop}
\label{wilson}
Every fiber of $\Upsilon$ has dimension $\sum_{l\in\BZ/n\BZ}(d_l^2+d_l)$.
\end{prop}

\proof First we assume that $\dim(\Upsilon^{-1}(D))=
\sum_{l\in\BZ/n\BZ}(d_l^2+d_l)$ for a colored divisor $D$ concentrated at
one point (with colored multiplicity). We will derive the general case
of the proposition from this particular case by induction in $\ul{d}$.
To this end, if a divisor
$D$ is a disjoint union of divisors $D^{(1)}$ and $D^{(2)}$ of degrees
$\ul{d}^{(1)}$ and
$\ul{d}^{(2)}$, and we know $\dim(\Upsilon^{-1}_{\ul{d}^{(1)}}(D^{(1)}))=
\sum_{l\in\BZ/n\BZ}((d^{(1)}_l)^2+d^{(1)}_l),\
\dim(\Upsilon^{-1}_{\ul{d}^{(2)}}(D^{(2)}))=
\sum_{l\in\BZ/n\BZ}((d^{(2)}_l)^2+d^{(2)}_l)$, we have to derive
$\dim(\Upsilon^{-1}(D))=\sum_{l\in\BZ/n\BZ}(d_l^2+d_l)$.

In effect, each space $V_l$ can be split into direct sum
$V_l=V^{(1)}_l\oplus V^{(2)}_l$, so that the endomorphism $A_l$ acquires the
block diagonal form $A_l=\left(\begin{array}{cc}
A_l^{(1)}&0\\
0&A_l^{(2)}
\end{array}\right)$, and $\on{Spec}A_l^{(1,2)}=D_l^{(1,2)}$. Note that
the space of such decompositions $V_l=V^{(1)}_l\oplus V^{(2)}_l$ is an open
subset in the product of two Grassmannians and has dimension
$2d_l^{(1)}d_l^{(2)}$. Now having written the matrices of
$(B_\bullet,p_\bullet,q_\bullet)$ in the block form according to our
decomposition, the equation $\mu=0$ takes the form
$$\left(\begin{array}{cc}
A_{l+1}^{(1)}&0\\
0&A_{l+1}^{(2)}
\end{array}\right)
\left(\begin{array}{cc}
B_l^{(11)}&B_l^{(12)}\\
B_l^{(21)}&B_l^{(22)}
\end{array}\right)-
\left(\begin{array}{cc}
B_l^{(11)}&B_l^{(12)}\\
B_l^{(21)}&B_l^{(22)}
\end{array}\right)
\left(\begin{array}{cc}
A_l^{(1)}&0\\
0&A_l^{(2)}
\end{array}\right)
+\left(\begin{array}{c}
p_{l+1}^{(1)}\\
p_{l+1}^{(2)}
\end{array}\right)
\left(\begin{array}{cc}
q_l^{(1)}&q_l^{(2)}
\end{array}\right)$$
$$=\left(\begin{array}{cc}
A_{l+1}^{(1)}B_l^{(11)}-B_l^{(11)}A_l^{(1)}+p_{l+1}^{(1)}q_l^{(1)}&
A_{l+1}^{(1)}B_l^{(12)}-B_l^{(12)}A_l^{(2)}+p_{l+1}^{(1)}q_l^{(2)}\\
A_{l+1}^{(2)}B_l^{(21)}-B_l^{(21)}A_l^{(1)}+p_{l+1}^{(2)}q_l^{(1)}&
A_{l+1}^{(2)}B_l^{(22)}-B_l^{(22)}A_l^{(2)}+p_{l+1}^{(2)}q_l^{(2)}
\end{array}\right)=
\left(\begin{array}{cc}
0&0\\
0&0
\end{array}\right).$$
In particular, we see that $(A^{(1)}_\bullet,B^{(11)}_\bullet,p^{(1)}_\bullet,
q^{(1)}_\bullet)$ (resp. $(A^{(2)}_\bullet,B^{(22)}_\bullet,p^{(2)}_\bullet,
q^{(2)}_\bullet)$) lies in $\sM_{\ul{d}^{(1)}}$
(resp. in $\sM_{\ul{d}^{(2)}}$). So by our induction hypothesis,
$\dim\{(A^{(1)}_\bullet,B^{(11)}_\bullet,p^{(1)}_\bullet,
q^{(1)}_\bullet):\ \on{Spec}A^{(1)}=D^{(1)}\}=
\sum_{l\in\BZ/n\BZ}((d^{(1)}_l)^2+d^{(1)}_l)$, and
$\dim\{(A^{(2)}_\bullet,B^{(22)}_\bullet,p^{(2)}_\bullet,
q^{(2)}_\bullet):\ \on{Spec}A^{(2)}=D^{(2)}\}=
\sum_{l\in\BZ/n\BZ}((d^{(2)}_l)^2+d^{(2)}_l)$. Recall that we also have
$2d_l^{(1)}d_l^{(2)}$ parameters for the choice of decomposition
$V_l=V_l^{(1)}\oplus V_l^{(2)}$. That already gives us the desired dimension
$\sum_{l\in\BZ/n\BZ}(d_l^2+d_l)$ altogether, and it only remains to prove
that the remaining equations have a unique solution in $B_l^{(12)},B_l^{(21)}$.
It follows from the fact that, say $A^{(2)}_{l+1}$ and $A^{(1)}_l$ having
disjoint spectra, do not admit any nontrivial intertwiners, and hence the
linear map $\Hom(V_l^{(1)},V_{l+1}^{(2)})\to\Hom(V_l^{(1)},V_{l+1}^{(2)}):\
B_l^{(21)}\mapsto A_{l+1}^{(2)}B_l^{(21)}-B_l^{(21)}A_l^{(1)}$ is an
isomorphism.

Since the statement of the proposition is obvious in case
$\sum_{l\in\BZ/n\BZ}d_l=1$, we have already proved the proposition in case
$D$ has no multiplicities (off-diagonal case). Moreover, we have proved that
$\Upsilon^{-1}(\BA^{\ul{d}}-\Delta)$ is smooth.

It remains to prove the proposition in the opposite extremal case when
$D$ is supported at one point. It does not matter, which point is it, so we
may and will assume it is 0. In other words, we assume that all the
endomorphisms $A_l$ are nilpotent. We follow the method of G.~Wilson in
his proof of~Lemma~1.11 of~\cite{w}. Suppose first that both $A_l$ and
$A_{l+1}$ are regular nilpotent. We choose bases in $V_l,V_{l+1}$ so that
the matrices of $A_l,A_{l+1}$ are Jordan blocks, and then we see that
the matrix of $A_{l+1}B_l-B_lA_l$ has the following property: for each
$i=1,\ldots,\min(d_l,d_{l+1})$ the sum of all elements in the $i$-th diagonal
(counting from the leftmost lowest corner) is 0. Now since
$A_{l+1}B_l-B_lA_l=-p_{l+1}q_l$ has rank 1, all these $\min(d_l,d_{l+1})$
diagonals
must vanish identically. It imposes the following restriction on the vector
$p_{l+1}$ and covector $q_l$ written down in our bases: the sum of numbers
of the last nonzero coordinate of $p_{l+1}$ and the first nonzero coordinate
of $q_l$ is greater than $\min(d_l,d_{l+1})$.
This means that the dimension of the
space of all possible collections $(p_{l+1},q_l)$ is at most
$\max(d_l,d_{l+1})$.

Recall that the dimension of the space of regular nilpotent matrices
$A_l$ (resp. $A_{l+1}$) is $d_l^2-d_l$ (resp. $d_{l+1}^2-d_{l+1}$).
Furthermore, for given $(A_l,A_{l+1},p_{l+1},q_l)$ the dimension of the
space of solutions of the linear equation $A_{l+1}B_l-B_lA_l=-p_{l+1}q_l$
equals (if it is not empty) the dimension of the space of intertwiners
$\on{Int}(A_l,A_{l+1})$, that is $\min(k,l)$. Altogether we obtain at most
$d_l^2+d_{l+1}^2-d_l-d_{l+1}+\min(d_l,d_{l+1})+\max(d_l,d_{l+1})$.
Summing up over all $l$
we obtain at most $\sum_{l\in\BZ/n\BZ}(d_l^2+d_l)$ parameters.

Now we turn to the general case and assume that the Jordan type of a
nilpotent matrix $A_l$ is given by a partition
$(\varkappa^{(l)}_1\geq\varkappa^{(l)}_2\geq\ldots)$.
Let $(\kappa^{(l)}_1\geq\kappa^{(l)}_2\geq\ldots)$ stand for the dual
partition. The space of all matrices $A_l$ of given type has dimension
$d_l^2-(\kappa^{(l)}_1)^2-(\kappa^{(l)}_2)^2-\ldots$.
We can choose some bases in the spaces $V_l$ so that the matrices of $A_l$
become the direct sums of Jordan blocks, and repeat the considerations of
two previous paragraphs blockwise. We come to the conclusion that the
dimension of the space of quadruples
$(A_\bullet,B_\bullet,p_\bullet,q_\bullet)$
such that the Jordan type of $A_l$ is
$(\varkappa^{(l)}_1\geq\varkappa^{(l)}_2\geq\ldots)$ is at most
$\sum_{l\in\BZ/n\BZ}(d_l^2-(\kappa^{(l)}_1)^2-(\kappa^{(l)}_2)^2-\ldots)+
\sum_{l\in\BZ/n\BZ}^{i,j\in\BN}\min(\varkappa^{(l)}_i,\varkappa^{(l+1)}_j)+
\sum_{l\in\BZ/n\BZ}\max(d_l,d_{l+1})$.
It is not hard to check (by induction in $\max_l(\kappa^{(l)}_1)$)
that this sum is at most
$\sum_{l\in\BZ/n\BZ}(d_l^2+d_l)$. On the other hand, the dimension of any
irreducible component of $\Upsilon^{-1}(\ul{d}\cdot0)$ cannot be less than
$\sum_{l\in\BZ/n\BZ}(d_l^2+d_l)$ since we have already seen that the generic
fiber of $\Upsilon$ has dimension $\sum_{l\in\BZ/n\BZ}(d_l^2+d_l)$.
This completes the proof of the proposition. \qed

\begin{cor}
\label{complete}
$\sM_{\ul{d}}$ is an irreducible reduced complete intersection in
$M^\Gamma_{\ul{d}}$.
\end{cor}

\proof The complete intersection property is clear from the comparison of
dimensions. It is also clear that $\dim\Upsilon^{-1}(\Delta)<
\sum_{l\in\BZ/n\BZ}(d_l^2+2d_l)$, and hence the closure of
$\Upsilon^{-1}(\BA^{\ul{d}}-\Delta)$ is the unique irreducible component
of $\sM_{\ul{d}}$. Finally, it was shown during the proof of
Proposition~\ref{wilson} that $\Upsilon^{-1}(\BA^{\ul{d}}-\Delta)$ is smooth,
and in particular, reduced. It follows from Proposition~5.8.5 of~\cite{ega} that
$\sM_{\ul{d}}$ is reduced. \qed

\begin{rem}
\label{trav}
{\em The subscheme $\Upsilon^{-1}(\ul{d}\cdot0)$ studied in the proof
of~Proposition~\ref{wilson} contains the nilcone $\CN_{\ul{d}}\subset
\sM_{\ul{d}}$. In the situation and notations of Example~\ref{sl2} the
nilcone $\CN_d\subset M_d$ is cut out by the equations
$a_1=\ldots=a_d=0=b_0=\ldots=b_{d-1}$. Equivalently, we require both
endomorphisms $A$ and $A+q\circ p$ to be nilpotent. Hence $\CN_d$
coincides with the {\em mirabolic nilpotent cone} introduced by
R.~Travkin in sections~1.3 and~3.2 of~\cite{t} (under the name of $Z$).
The beautiful geometry of $\CN_d$ studied in {\em loc. cit.} suggests
that $\CN_{\ul{d}}$ might be an interesting object in itself.}
\end{rem}

\subsection{Proof of Theorem~\ref{bij}.a)}
\label{big}
The categorical quotient $\fZ_{\ul{d}}$ inherits the properties of being
reduced and irreducible from $\sM_{\ul{d}}$. To prove the normality of
$\fZ_{\ul{d}}$ we will use~Corollary~7.2 of~\cite{cb}. To this end we will
exhibit a normal open subscheme $U\subset\fZ_{\ul{d}}$ such that its
complement $Y\subset\fZ_{\ul{d}}$ is of codimension 2, and
$\Psi^{-1}(Y)$ is of codimension 2 in $\sM_{\ul{d}}$. Here
$\Psi:\ \sM_{\ul{d}}\to\fZ_{\ul{d}}$ is the natural projection.
Note that $\sM_{\ul{d}}$ is Cohen-Macaulay (being a complete intersection),
in particular, it has property $(S_2)$. So all the conditions of
{\em loc. cit.} will be verified, and it will guarantee the normality of
$\fZ_{\ul{d}}$.

To construct $U\subset\fZ_{\ul{d}}$ note that the morphism
$\Upsilon:\ \sM_{\ul{d}}\to\BA^{\ul{d}}$ evidently factors as
$\sM_{\ul{d}}\stackrel{\Psi}{\to}\fZ_{\ul{d}}\stackrel{\Phi}{\to}\BA^{\ul{d}}$
for a uniquely defined morphism $\Phi$. We introduce an open subset
$\hat U\subset\BA^{\ul{d}}$ formed by all the colored configurations where
at most 2 points collide. We set $U:=\Phi^{-1}(\hat U)$.

Evidently, the complement $\BA^{\ul{d}}-\hat U$ is of codimension 2
in $\BA^{\ul{d}}$, and so the codimension conditions on $U$ are satisfied.
It remains to prove that $U$ is normal. The argument of the first part of the
proof of Proposition~\ref{wilson} shows that after an \'etale base change
in a formal neighbourhood of a point in $\hat U$
(an ordering of distinct points in a configuration in $\hat U$), both
$\Upsilon^{-1}(\hat U)$ and $\Phi^{-1}(\hat U)=U$ decompose into a direct
product of a smooth scheme, and a scheme of one of
Examples~\ref{sl2},~\ref{sl3},~\ref{hatsl2}. Namely, Example~\ref{sl2}
occurs if two points of the same color collide; Example~\ref{sl3} occurs
if two points of different colors collide, and $n>2$; finally,
Example~\ref{hatsl2} occurs if two points of different colors collide, and
$n=2$. Obviously, all the schemes of the above Examples are normal.
As normality is stable under the \'etale base change and the formal completion,
the proof of~Theorem~\ref{bij}.a) is complete. \qed

\subsection{Proof of Theorem~\ref{bij}.b)}
\label{d-vo}
To prove b), we recall the stratification of
$Z^{\ul{d}}$ introduced in section~6.6 of~\cite{fgk}. It obviously coincides
with the stratification of $\fZ_{\ul{d}}$ introduced in~\ref{strat}.
In particular, we have a bijection between the sets of $\BC$-points of
$Z^{\ul{d}}$ and $\fZ_{\ul{d}}$.
Moreover, for a $\BC$-point $s$ in a stratum of $Z^{\ul{d}}$,
and the same named corresponding point in the corresponding stratum
of $\fZ_{\ul{d}}$, the (reduced) fibers $\pi^{-1}(s)\subset\CP_{\ul{d}}
\supset\varpi^{-1}(s)$ coincide. In effect, they are both formed by all
the parabolic sheaves with given saturation and defect in terminology of
{\em loc. cit.} Now the existence of $\eta$ follows from normality of
$\fZ_{\ul{d}}$ e.g. by the argument in the proof of~Proposition~2.14
of~\cite{bf2}. Theorem~\ref{bij} is proved. \qed

\begin{conj}
The morphism $\eta:\ \fZ_{\ul{d}}\to Z^{\ul{d}}$ is an isomorphism.
\end{conj}

\begin{rem}
{\em This conjecture was proved in~\cite{BF11}.}
\end{rem}

\subsection{The character of $\BC[\fZ_{\ul{d}}]$}
\label{character}
Corollary~\ref{complete} gives rise to a formula for the character of
$\BC[\fZ_{\ul{d}}]$. Let $T$ stand for the Cartan torus of $GL(W)$ which
acts on the basis vector $w_k$ via the character $\bt_k,\ k=1,\ldots,n$.
Thus $\bT:=\BC^*\times\BC^*\times T$ acts on $\CP_{\ul{d}}$ via the
action of the first (resp. second) copy of $\BC^*$ on $\bC$ (resp. on $\bX$)
via the character $\bv$ (resp. $\bfu$), see~\ref{PS}. The relation to
the notations of~\cite{bf} is as follows: $\bt_k=t_k^2,\ \bv=v^2,\ \bfu=u^2$.
Now the character of $\BC[\fZ_{\ul{d}}]$ as a $\bT$-module is a formal
power series in $\bt_1,\ldots,\bt_n,\bfu,\bv$ which is actually a Laurent
expansion of a rational function to be denoted by $F_{\ul{d}}$.

To calculate $F_{\ul{d}}$ we note that the action of $\bT$ on
$\BC[\fZ_{\ul{d}}]$ arises from the following action of $\bT$ on the
symmetric algebra $\BC[M^\Gamma_{\ul{d}}]$. Let us choose a base
$v_{l,1},\ldots,v_{l,d_l}$ in $V_l$, and denote the corresponding matrix
elements of $A_l$ (resp. $B_l,p_l,q_l$) by
$(A_{l}^{(ij)})_{1\leq i\leq d_l}^{1\leq j\leq d_l}$
(resp. $(B_{l}^{(ij)})_{1\leq i\leq d_l}^{1\leq j\leq d_{l+1}},\
(p_l^{(i)})_{1\leq i\leq d_l},\
(q_l^{(i)})_{1\leq i\leq d_l}$). Moreover, let us denote by
$\sT$ the Cartan torus of $G_{\ul{d}}$ acting on a base vector
$v_{l,i}$ via the character $\st_{l,i}$. Then the eigenvalues of the
$\bT\times\sT$-action on the generators of $\BC[M^\Gamma_{\ul{d}}]$
are as follows: $A_{l}^{(ij)}:\ \bv\st_{l,i}\st_{l,j}^{-1},\qquad
B_{l}^{(ij)}:\ \bfu^{\delta_{0,l}}\st_{l,i}\st_{l+1,j}^{-1},\qquad
p_l^{(i)}:\ \bfu^{\delta_{1,l}}\bv\bt_{l-1}\st_{l,i}^{-1},\qquad
q_l^{(i)}:\ \bt_l^{-1}\st_{l,i}$.

The character of the $\bT\times\sT$-action on the symmetric algebra
$\BC[M^\Gamma_{\ul{d}}]$ equals $S_{\ul{d}}:=$
\begin{multline*}\prod_{\substack{l\in\BZ/n\BZ\\1\leq i,j\leq d_l}}(1-\bv\st_{l,i}\st_{l,j}^{-1})^{-1}
\prod_{\substack{l\in\BZ/n\BZ\\1\leq i\leq d_l\\
1\leq j\leq d_{l+1}}}
(1-\bfu^{\delta_{0,l}}\st_{l,i}\st_{l+1,j}^{-1})^{-1}\times\\ \times
\prod_{\substack{l\in\BZ/n\BZ\\1\leq i\leq d_l}}
(1-\bfu^{\delta_{1,l}}\bv\bt_{l-1}\st_{l,i}^{-1})^{-1}
\prod_{\substack{l\in\BZ/n\BZ\\1\leq i\leq d_l}}(1-\bt_l^{-1}\st_{l,i})^{-1}.
\end{multline*}
The space of equations cutting out $\sM_{\ul{d}}\subset M^\Gamma_{\ul{d}}$
has a natural base consisting of the matrix elements
$(E_{l}^{(ij)})_{1\leq i\leq d_l}^{1\leq j\leq d_{l+1}}$ of the matrices
$A_{l+1}B_l-B_lA_l+p_{l+1}q_l$. The eigenvalue of the $\bT\times\sT$-action
on $E_{l}^{(ij)}$ is $\bfu^{\delta_{0,l}}\bv\st_{l,i}\st_{l+1,j}^{-1}$.
The (graded) character of the $\bT\times\sT$-action on the external algebra
generated by $\{(E_{l}^{(ij)})_{1\leq i\leq d_l}^{1\leq j\leq d_{l+1}}\}$
equals
$\Lambda_{\ul{d}}:=\prod\limits_{\substack{l\in\BZ/n\BZ\\1\leq i\leq d_l\\
1\leq j\leq d_{l+1}}}
(1-\bfu^{\delta_{0,l}}\bv\st_{l,i}\st_{l+1,j}^{-1}).$
According to Corollary~\ref{complete}, the character of the
$\bT\times\sT$-action on $\BC[\sM_{\ul{d}}]$ equals
$S_{\ul{d}}\Lambda_{\ul{d}}$. Finally, the character $F_{\ul{d}}$ of the
$\bT$-action on $\BC[\fZ_{\ul{d}}]=\BC[\sM_{\ul{d}}]^{G_{\ul{d}}}$ equals
$(1,S_{\ul{d}}\Lambda_{\ul{d}})_\sT$ where $(\cdot,\cdot)_\sT$ is the scalar
product of $G_{\ul{d}}$-characters.

\section{Hamiltonian reduction}

\subsection{Poisson structure on Laumon and Drinfeld spaces}
\label{poi}
Recall that $\CP^\circ_{\ul{d}}\subset\CP_{\ul{d}}$ stands for the open
subset of locally free parabolic sheaves. According to section~5 of~\cite{fgk},
$\CP^\circ_{\ul{d}}$ is the moduli space of based maps
of degree $\ul{d}$ from $(\bC,\infty_\bC)$
to the Kashiwara flag scheme of the affine Lie algebra $\widehat{\fsl}(n)$.
According to section~1 of~\cite{bfg}, such a moduli space of based maps
is defined for any Kac-Moody Lie algebra $\fg$; let us denote it by
$\CP^\circ_{\fg,\ul{d}}$. In case $\fg$ is a simple Lie algebra, a symplectic
structure on $\CP^\circ_{\fg,\ul{d}}$ was constructed in~\cite{fkmm}.
This construction applies {\em verbatim} to $\CP^\circ_{\fg,\ul{d}}$ for any
Kac-Moody Lie algebra $\fg$, in particular for $\fg=\widehat{\fsl}(n)$, and
provides $\CP^\circ_{\ul{d}}=\CP^\circ_{\widehat{\fsl}(n),\ul{d}}$ with a
symplectic structure $\Omega$, and corresponding Poisson bracket
$\{\cdot,\cdot\}_K$.
F.~Bottacin~\cite{b} has generalized this Poisson bracket to the moduli
spaces of stable parabolic locally free sheaves on arbitrary smooth projective
surfaces.

\begin{lem}
\label{extension}
The Poisson structure $\{\cdot,\cdot\}_K$ on $\CP^\circ_{\ul{d}}$ extends
uniquely to the same named Poisson structure on $\CP_{\ul{d}}$.
\end{lem}

\proof
The complement $\CP_{\ul{d}}-\CP^\circ_{\ul{d}}$ is a union of
Cartier divisors (see e.g. section~11 of~\cite{bfg}). In the \'etale
$(x,y)$-coordinates of section~3.3 of~\cite{fkmm}, these divisors are
just the zero divisors of $y$-coordinates. Now the explicit formula
of~Proposition~2 of {\em loc. cit.} shows that our bracket $\{\cdot,\cdot\}_K$
extends regularly through the generic points of these divisors.
Since $\CP_{\ul{d}}$ is smooth, and the bivector field $\{\cdot,\cdot\}_K$ is
regular off codimension 2, it is regular everywhere. \qed

\begin{cor}
\label{extend}
The Poisson structure $\{\cdot,\cdot\}_K$ on
$\CP^\circ_{\ul{d}}\subset\fZ_{\ul{d}}$ extends
uniquely to the same named Poisson structure on $\fZ_{\ul{d}}$.
\end{cor}

\proof
The (reduced) fibers of the resolution $\pi:\ \CP_{\ul{d}}\to\fZ_{\ul{d}}$ were
already identified with the (reduced) fibers of the resolution
$\varpi:\ \CP_{\ul{d}}\to Z^{\ul{d}}$ in~\ref{d-vo}. The latter fibers
are described in section~6 of~\cite{fgk}, in particular they are connected.
Due to normality of $\fZ_{\ul{d}}$, the algebra of functions
$\BC[\fZ_{\ul{d}}]$ coincides with the algebra $\BC[\CP_{\ul{d}}]$.
So the Poisson bracket on $\BC[\fZ_{\ul{d}}]$ is obtained just as global
sections of the Poisson bracket on $\CP_{\ul{d}}$. \qed

\subsection{Separating variables}
\label{split}
The Poisson bracket $\{\cdot,\cdot\}_K$ on $\CP^\circ_{\ul{d}}$ acquires a very
simple form in the \'etale $(x,y)$-coordinates of section~3.3 of~\cite{fkmm}.
We recall these coordinates in the quiver description of~\ref{queer}.
We consider an open subset $U\subset\CP^\circ_{\ul{d}}\subset\fZ_{\ul{d}}$
formed by the classes of (stable and costable) quadruples
$(A_l,B_l,p_l,q_l)_{l\in\BZ/n\BZ}$ such that all the endomorphisms $A_l$
have simple and disjoint spectra. We order their eigenvalues some way, and
denote them by $(x_{l1},\ldots,x_{l,d_l})_{l\in\BZ/n\BZ}$. Furthermore,
following Example~\ref{sl2}, for $r\in\BN$ we denote by $b_{l,r}$
the composition $q_l\circ A_{l}^r\circ p_l$. Moreover, for $j\in\BN$, we
denote by $\sigma_j$ the $j$-th elementary symmetric function (in particular,
$\sigma_0=1$). Finally, for $1\leq r\leq d_l$ we define
$y_{l,r}:=\sum_{s=0}^{d_l-1}(-1)^sb_{l,d_l-1-s}
\sigma_s(x_{l1},\ldots,x_{l,r-1},x_{l,r+1},\ldots,x_{l,d_l})$.

\begin{prop}
\label{kuznecov}
$\{x_{l,r},x_{k,s}\}_K=0=\{y_{l,r},y_{l,s}\}_K;\
\{x_{l,r},y_{k,s}\}_K=\delta_{lk}\delta_{rs}y_{k,s};\
\{y_{l,r},y_{k,s}\}_K=c_{lk}\frac{y_{l,r}y_{k,s}}{x_{l,r}-x_{k,s}}$ for $k\ne l$,
where $(c_{lk})_{k,l\in\BZ/n\BZ}$ stands for the Cartan matrix of
$\widehat{\fsl}(n)$.
\end{prop}

\proof
We only have to check that our coordinates $x_{l,r},y_{k,s}$ coincide with
what is denoted by $x^r_l,y^s_k$ in~\cite{fkmm}, and then to
apply~Proposition~2 of {\em loc. cit.} (whose proof applies {\em verbatim}
to the case $\fg=\widehat{\fsl}(n)$). The matching of $x_{l,r},y_{k,s}$ with
$x^r_l,y^s_k$ clearly reduces to the case of $\on{SL}(2)$ of Example~\ref{sl2}.
So to simplify the notations, we denote $d_l$ by $d$, and $A_l$ by $A$,
and $x_{l,r}$ by $x_r$,
and $y_{l,r}$ by $y_r$, and $b_{l,s}$ by $b_s$. Note that the coordinates
$a_m$ of~\ref{sl2} are just $a_m=x_1^m+\ldots+x_d^m$.

Recall that $\fZ_d=Z^d$ naturally identifies with the space of pairs
$\{(P(z),Q(z))\}$ of polynomials in $z$ such that $\deg(P(z))=d$, and
the leading coefficient of $P(z)$ is 1, and $\deg(Q(z))<d$
(see section~1.2 of~\cite{fkmm}). The coordinates
$x^r,\ 1\leq r\leq d$, of {\em loc. cit.} are just the roots of $P(z)$, while
$y^r=Q(x^r)$. Evidently, $P(z)$ is nothing else than the characteristic
polynomial of the endomorphism $A$, so we can identify $x^r=x_r$. Let us
redenote $y^s$ by $y'_s$ to avoid a confusion of upper indices with powers.
Then it remains to prove that $y'_s=c_sy_s$ for some constant $c_s$.

Note that $\BC[\fZ_d]$ is bigraded so that $\deg(a_r)=(0,r),\ \deg(b_s)=(1,s)$.
This grading arises from the action of $\BC^*\times\BC^*$ on $\fZ_d=Z^d$. From
the point of view of Zastava $Z^d$, the first copy of $\BC^*$ is acting on
$\bC$ by ``loop rotations'', while the second copy of $\BC^*$ is acting as
the Cartan torus in $\on{SL}(2)$ (corresponding to the decomposition
$W=\langle w_1\rangle\oplus\langle w_2\rangle$). Thus, if we write
$P(z)=z^d+e_1z^{d-1}+\ldots+e_d,\ Q(z)=f_0z^{d-1}+\ldots+f_{d-1}$, then
$e_r$ has bidegree $(0,r)$, while $f_s$ has bidegree $(1,s)$.
Hence, up to a multiplicative constant, we have $f_s=b_s+\sum_{1\leq r\leq s}
\phi_{s,r}b_{s-r}$ where $\phi_{s,r}$ is a symmetric degree $r$ polynomial
in $x_1,\ldots,x_d$.

\begin{lem}
\label{len}
Up to a multiplicative constant, we have

a)$y'_s=y_s$;\qquad b)$f_s=b_s+\sum_{1\leq r\leq s}e_rb_{s-r}$;\qquad
c)$\frac{Q(z)}{P(z)}=\sum_{r=0}^\infty b_rz^{-1-r}$.
\end{lem}

\proof
b) and c) are clearly equivalent. Moreover, a) is equivalent to b), i.e. to
$\phi_{s,r}=e_r$. In effect, the equality $y'_s=y_s$
is equivalent by Lagrange interpolation to
$$Q(z)=\sum_{1\leq r\leq d}y_r\prod_{m\ne r}(z-x_m)(x_r-x_m)^{-1}=$$
$$=\sum_{1\leq r\leq d}\left(\sum_{s=0}^{d-1}(-1)^sb_{d-1-s}
\sigma_s(x_1,\ldots,x_{r-1},x_{r+1},\ldots,x_d)\right)
\prod_{m\ne r}(z-x_m)(x_r-x_m)^{-1},$$
and hence $f_s=b_s+\sum_{1\leq r\leq s}e_rb_{s-r}$.

To prove a), by unique factorization in the polynomial ring
$\BC[x_1,\ldots,x_d,b_0,\ldots,b_{d-1}]=
\BC[x_1,\ldots,x_d,f_0,\ldots,f_{d-1}]$, it suffices to see that
$y_1\ldots y_d=cy'_1\ldots y'_d$ for some constant $c$, that is
$$\prod_{1\leq r\leq d}\left(\sum_{s=0}^{d-1}(-1)^sb_{d-1-s}
\sigma_s(x_1,\ldots,x_{r-1},x_{r+1},\ldots,x_d)\right)=c
\prod_{1\leq r\leq d}\left(\sum_{s=0}^{d-1}f_sx_r^{d-1-s}\right).$$
Now $y'_1\ldots y'_d$ is an equation (resultant of $P(z),Q(z)$)
of the boundary divisor $BZ^d:=
Z^d-\CP^\circ_d$ (defined uniquely up to a multiplicative constant).
It remains to prove that $y_1\ldots y_d$ is also an equation of
$BZ^d\subset Z^d$.

To this end, note that $Z^d=\fZ_d=\mu^{-1}(0)^s/GL(d)$ where $\mu^{-1}(0)^s$
stands for the open subset formed by all the stable triples $(A,p,q)$, i.e.
such that $V$ has no proper $A$-invariant subspaces containing $\on{Im}(p)$.
The preimage of $BZ^d$ in $\mu^{-1}(0)^s$ consists of stable but
noncostable triples $(A,p,q)$, i.e. such that $V$ has an $A$-invariant
vector $v$ contained in $\Ker(q)$. In case $A$ has a simple spectrum
$\{x_1,\ldots,x_d\}$ with corresponding eigenvectors $\{v_1,\ldots,v_d\}$,
the vector $v$ can only be one of $\{v_1,\ldots,v_d\}$.
We have $v=v_m$ if and only if the vector $^t(b_0,\ldots,b_{d-1})$ lies
in the span of the vectors $^t(1,x_r,\ldots,x_r^{d-1})_{r\ne m}$, i.e.

$$\det\left(\begin{array}{cccccccc}
1&1&\ldots&1&1&\ldots&1&b_0\\
x_1&x_2&\ldots&x_{m-1}&x_{m+1}&\ldots&x_d&b_1\\
\vdots&\vdots&\ddots&\vdots&\vdots&\ddots&\vdots&\vdots\\
x_1^{d-1}&x_2^{d-1}&\ldots&x_{m-1}^{d-1}&x_{m+1}^{d-1}&\ldots&x_d^{d-1}&b_{d-1}
\end{array}\right)=0.$$
This determinant is obviously divisible by the Vandermonde determinant
in the variables $(x_1,\ldots,x_{m-1},x_{m+1},\ldots,x_d)$,
and the ratio is equal to $$\sum_{s=0}^{d-1}(-1)^sb_{d-1-s}
\sigma_s(x_1,\ldots,x_{m-1},x_{m+1},\ldots,x_d)=y_m.$$
We conclude that $y_1\ldots y_d$ is an equation of $BZ^d\subset Z^d$.
The lemma is proved along with Proposition~\ref{kuznecov}. \qed

\subsection{Classical Hamiltonian reduction. $\fsl_2$ case.}
Let $V=\CC^d$ be a finite-dimensional vector space. Consider the Lie algebra $\fa:=(\fgl(V)\ltimes V)\oplus (\fgl(V)\ltimes V^*)$ (the semidirect product is with respect to the tautological action of $\fgl(V)$ on $V$ and $V^*$). Let $\fgl(V)_{\diag}$ be the diagonal $\fgl(V)$ inside $\fgl(V)\oplus \fgl(V)\subset\fa$ and $\pi:\fa^*\to\fgl(V)_{\diag}^*$ be the projection.

The Drinfeld Zastava space $\fZ_d$ is the categorical quotient $(\fgl(V)\oplus V\oplus V^*)/GL(V)$ and hence can be identified with the Hamiltonian reduction $\fa^*//GL(V)_{\diag}=\pi^{-1}(0)/GL(V)_{\diag}$. This provides a natural Poisson bracket $\{\cdot,\cdot\}$ on $\fZ_d$.

Let us write this explicitly. Let $e_{ij}, e_{ij}', q_i, p_i$, where $1\le i,j\le d$, be the basis of $\fa$ such that

\begin{equation}
[e_{ij},e_{kl}]=\delta_{jk}e_{il}-\delta_{il}e_{jk},\quad
[e_{ij}',e_{kl}']=\delta_{jk}e_{il}'-\delta_{il}e_{jk}',
\end{equation}

\begin{equation}
[e_{ij},e_{kl}']=[e_{ij},p_k]=[e_{ij}',q_k]=[p_k,q_l]=0,
\end{equation}

\begin{equation}
[e_{ij},q_k]=\delta_{jk}q_i,\quad
[e_{ij}',p_k]=-\delta_{ki}p_j.
\end{equation}
I.e. $e_{ij}$ (resp. $e_{ij}'$) is the standard basis of the first copy of $\fgl(V)$ (resp. second copy of $\fgl(V)$) and $q_i, p_i$ are the bases of $V$ and $V^*$, respectively.

The coordinate ring of the Hamiltonian reduction $\fZ_d=\fa^*//\fgl(V)_{\diag}$ is

$$
\CC[\fZ_d]=\CC[\fa^*//\fgl(V)_{\diag}]=
\left(\CC[e_{ij}, e_{ij}', q_i, p_i]/(e_{ij}+e_{ij}')\right)^{\fgl(V)_{\diag}}=\CC[e_{ij}, q_i, p_i]^{\fgl(V)}.
$$

\begin{rem} One can also treat $\CC[\fa^*//\fgl(V)_{\diag}]$ as $\CC[e_{ij}', q_i, p_i]^{\fgl(V)}$.
\end{rem}

\subsection{Calculation of Poisson brackets}

According to the classical invariant theory, the algebra $\CC[\fZ_d]=\CC[e_{ij}, q_i, p_i]^{\fgl(V)}$ is generated by the following polynomial invariants $$a_r:=\Tr A^r=\sum\limits_{i_1,\ldots,i_r} e_{i_1i_2}e_{i_2i_3}\ldots e_{i_ri_1},\ r=1,\ldots, d;$$ $$b_s:=\langle p, A^sq\rangle=\sum\limits_{i_1,\ldots,i_{s+1}} p_{i_1}e_{i_1i_2}e_{i_2i_3}\ldots e_{i_si_{s+1}}q_{i_{s+1}},\ s=0,\ldots,d-1.$$

Introduce the following notation:

$$
e_{ij}^{(r)}:=\sum\limits_{i_1,\ldots,i_{r-1}} e_{ii_1}e_{i_1i_2}e_{i_2i_3}\ldots e_{i_{r-1}j}.
$$
We also set $e_{ij}^{(0)}=\delta_{ij}$.

We will use the following relations:
\begin{lem}
\begin{eqnarray} \sum\limits_{i}e_{ki}^{(r)}e_{ij}^{(s)}=e_{kj}^{(r+s)};\\
\{e_{kl},e_{ij}^{(r)}\}=\delta_{il}e_{kj}^{(r)}-\delta_{kj}e_{il}^{(r)};\\
\{e_{kl},\sum\limits_{j=1}^de_{ij}^{(r)}q_j\}=\delta_{il}\sum\limits_{j=1}^de_{kj}^{(r)}q_j;\\
\{e_{kl}^{(s)},\sum\limits_{j=1}^de_{ij}^{(r)}q_j\}=\sum\limits_{t=1}^{s}\sum\limits_{j=1}^de_{kj}^{(r+t-1)}q_je_{il}^{(s-t)}.
\end{eqnarray}
\end{lem}

\begin{proof}
Straightforward.
\end{proof}

\begin{prop}\label{prop-poisson-relations}
\begin{eqnarray}
\{a_r,a_s\}=0;\\
\{a_r,b_s\}=rb_{r+s-1};\\
\{b_r,b_s\}=\sum\limits_{m=s}^{r-1}b_{m}b_{r+s-m-1}.
\end{eqnarray}
\end{prop}

\begin{proof} The first is obvious. Let us prove the second one:
\begin{multline*}
\{a_r,b_s\}=\{\sum\limits_{k} e_{kk}^{(r)},\sum\limits_{ij} p_{i}e_{ij}^{(s)}q_{j}\}=\\
=\sum\limits_{t=1}^r\sum\limits_{k,i,j}p_{i}e_{kj}^{(t+s-1)}q_je_{ik}^{(r-t)}=r\sum\limits_{ij} p_{i}e_{ij}^{(r+s-1)}q_{j}=rb_{r+s-1}.
\end{multline*}

And the third one:

\begin{multline*}
\{b_r,b_s\}=\{\sum\limits_{k,l} p_{k}e_{kl}^{(r)}q_{l},\sum\limits_{i,j} p_{i}e_{ij}^{(s)}q_{j}\}=\\
=\sum\limits_{t=1}^r\sum\limits_{k,l,i,j}p_kp_{i}e_{kj}^{(t+s-1)}q_je_{il}^{(r-t)}q_l-
\sum\limits_{m=0}^{s-1}\sum\limits_{k,l,i,j}p_kp_{i}e_{il}^{(m)}q_le_{kj}^{(r+s-m-1)}q_j=\\=
\sum\limits_{m=s}^{r-1}b_{m}b_{r+s-m-1}.
\end{multline*}

\end{proof}

Let $x_1,\ldots, x_d$ and $y_1,\ldots, y_d$ be the following \'etale coordinates on $\fZ_d$:

\begin{eqnarray*}
a_r=\sum\limits_{i=1}^d x_i^r,\quad y_i=\sum\limits_{r=0}^{d-1} (-1)^{r}\sigma_r(x_1,\ldots, \widehat{x_i},\ldots, x_d)b_{d-1-r},
\end{eqnarray*}
where $\sigma_r$ stands for the elementary symmetric function of degree $r$.

\begin{prop}\label{kuzn-sl2}
We have $\{x_i,x_j\}=0=\{y_i,y_j\}$ and $\{x_i,y_j\}=\delta_{ij}y_j$.
\end{prop}

\begin{proof}
$\{x_i,x_j\}=0$ is obvious.

Let $\ft$ be the Cartan subalgebra of $\fgl(V)$ generated by $e_{ii}$ with $i=1,\ldots,d$. Note that each $GL(V)$-invariant function on $\fgl(V)\oplus V\oplus V^*$ is uniquely determined by its restriction to the $\ft$-invariant subspace $S:=\ft\oplus V\oplus V^*$. We have $$
y_j|_S=\prod\limits_{k\ne j}(x_j-x_k)p_jq_j.
$$
Hence
$$\{x_i,y_j\}|_S=\ad(dx_i)(y_j)|_S=\ad(e_{ii})(y_j|_S)=\delta_{ij}y_j|_S.$$
This implies $\{x_i,y_j\}=\delta_{ij}y_j$.

We have
\begin{multline*}
\{y_i,y_j\}=\\=\{\sum\limits_{r=0}^{d-1} (-1)^{r}\sigma_r(x_1,\ldots, \widehat{x_i} ,\ldots, x_d)b_{d-1-r},\sum\limits_{r=0}^{d-1} (-1)^{r}\sigma_r(x_1,\ldots, \widehat{x_j} ,\ldots, x_d)b_{d-1-r}\}=\\
=(\sum\limits_{r=0}^{d-1} (-1)^{r}\frac{\partial}{\partial x_j}\sigma_r(x_1,\ldots, \widehat{x_i} ,\ldots, x_d)b_{d-1-r})y_j-\\
-(\sum\limits_{r=0}^{d-1} (-1)^{r}\frac{\partial}{\partial x_i}\sigma_r(x_1,\ldots, \widehat{x_j} ,\ldots, x_d)b_{d-1-r})y_i+\\
+\sum\limits_{r=0}^{d-1}\sum\limits_{s=0}^{d-1} (-1)^{r+s}\sigma_r(x_1,\ldots, \widehat{x_i} ,\ldots, x_d)\sigma_s(x_1,\ldots, \widehat{x_j} ,\ldots, x_d)\{b_{d-1-r},b_{d-1-s}\}.
\end{multline*}

Set $\sigma_r:=\sigma_r(x_1,\ldots, \widehat{x_i}, \widehat{x_j},\ldots, x_d)$. Applying the equation $\sigma_r(x_1,\ldots, \widehat{x_i} ,\ldots, x_d)=\sigma_r+x_j\sigma_{r-1}$, we obtain
\begin{multline*}
\{y_i,y_j\}=\sum\limits_{r=0}^{d-1}\sum\limits_{s=0}^{d-1} (-1)^{r+s}\sigma_{r-1}(\sigma_s+x_i\sigma_{s-1})b_{d-1-r}b_{d-1-s}-\\
-\sum\limits_{r=0}^{d-1}\sum\limits_{s=0}^{d-1} (-1)^{r+s}\sigma_{r-1}(\sigma_s+x_j\sigma_{s-1})b_{d-1-r}b_{d-1-s}+\\
+\sum\limits_{r=0}^{d-1}\sum\limits_{s=0}^{d-1} (-1)^{r+s}(\sigma_r+x_j\sigma_{r-1})(\sigma_s+x_i\sigma_{s-1})\{b_{d-1-r},b_{d-1-s}\}=\\
=\sum\limits_{r=0}^{d-2}\sum\limits_{s=0}^{d-2} (-1)^{r+s}(x_i-x_j)\sigma_{r}\sigma_{s}b_{d-2-r}b_{d-2-s}-\\
-\sum\limits_{r=0}^{d-1}\sum\limits_{s=r}^{d-1} (-1)^{r+s}((\sigma_{r}+x_i\sigma_{r-1})(\sigma_{s}+x_j\sigma_{s-1})-(\sigma_{r}+x_j\sigma_{r-1})(\sigma_{s}+x_i\sigma_{s-1}))\{b_{d-1-r},b_{d-1-s}\}=\\
=\sum\limits_{r=0}^{d-2}\sum\limits_{s=0}^{d-2} (-1)^{r+s}(x_i-x_j)\sigma_{r}\sigma_{s}b_{d-2-r}b_{d-2-s}-\\
-\frac{1}{2}\sum\limits_{r=0}^{d-2}\sum\limits_{s=0}^{d-2} (-1)^{r+s}(x_i-x_j)\sigma_{r}\sigma_{s}(\{b_{d-1-r},b_{d-2-s}\}-\{b_{d-2-r},b_{d-1-s}\})=0.
\end{multline*}

\end{proof}

\begin{cor} $\{\cdot,\cdot\}=\{\cdot,\cdot\}_K$ on $\fZ_d$.
\end{cor}

\subsection{Classical Hamiltonian reduction. General case.}
We fix an $n$-tuple $\ul{d}=(d_1,\ldots,d_n)$ of nonnegative integers. Let $V_{\ul{d}}=\bigoplus\limits_{l=1,\dots,n}V_l=\bigoplus\limits_{l=1,\dots,n}\BC^{d_l}$ and $\fgl(V_{\ul{d}}):=\bigoplus\limits_{l=1,\dots,n}\fgl_{d_l}$.

To present Zastava spaces as Hamiltonian reduction, we ``triangulate'' the chainsaw quiver in the following way:

$$
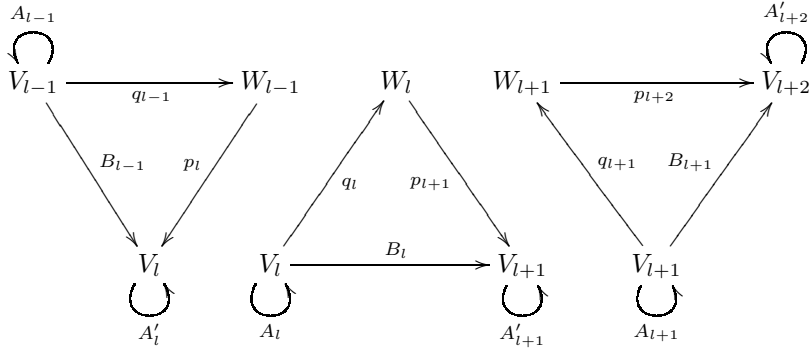
\xymatrix{
V_{l-1} \ar@(ur,ul)[]_{A_{l-1}} \ar[rdd]^{B_{l-1}} \ar[rr]_{q_{l-1}} &
& W_{l-1} \ar[ldd]_{p_l} & W_l \ar[rdd]_{p_{l+1}} &
W_{l+1} \ar[rr]_{p_{l+2}} && V_{l+2} \ar@(ur,ul)[]_{A'_{l+2}}       \\
&&&&&&\\
& V_l \ar@(dl,dr)[]_{A'_l} &
V_l \ar@(dl,dr)[]_{A_l} \ar[rr]^{B_l} \ar[ruu]_{q_l} &
& V_{l+1} \ar@(dl,dr)[]_{A'_{l+1}}
& V_{l+1} \ar@(dl,dr)[]_{A_{l+1}} \ar[ruu]^{B_{l+1}} \ar[luu]_{q_{l+1}} &
}$$

For a pair of vector spaces $V_l,V_{l+1}$ define the following $2$-step nilpotent Lie algebra: $$\fn(V_l,V_{l+1}):=V_l\oplus V_{l+1}^*\oplus (V_l\otimes V_{l+1}^*),$$ where the space $V_l\otimes V_{l+1}^*$ is central, $[V_l,V_l]=[V_{l+1}^*,V_{l+1}^*]=0$, and for $v\in V_l,\ w^\vee\in V_{l+1}^*$ one has $[v,w^\vee]=v\otimes w^\vee$.

To define the Poisson structure, we attach to each triangle of our graph the following Lie algebra
$$\fa_l:=(\fgl(V_l)\oplus \fgl(V_{l+1}))\rtimes \fn(V_l,V_{l+1})
$$ (the semidirect sum is with respect to the tautological action of $\fgl(V_l)$ on $V_l$ and $\fgl(V_{l+1})$ on $V_{l+1}^*$).

Consider the Lie algebra
$$\fa_{\ul{d}}:=\bigoplus\limits_{l\in\BZ/n\BZ}\fa_l=\bigoplus\limits_{l\in\BZ/n\BZ}(\fgl(V_l)\oplus \fgl(V_{l+1}))\rtimes \fn(V_l,V_{l+1})
$$

The coadjoint representation of $\fa_{\ul{d}}$ is the space $\fa_{\ul{d}}^*=\{(A_l,A_l',B_l,p_l,q_l)_{l\in\BZ/n\BZ}\}$, where
$$A_l\in\End(V_l),\quad A_l'\in\End(V_l),\quad B_l\in\Hom(V_l,V_{l+1}),\quad p_l\in V_l,\quad q_l\in V_l^*.$$
Let $\{e_{l,ij},e_{l,ij}'\}_{l\in\BZ/n\BZ,\ 1\le i,j\le d_l}$ be the coordinates on the $2$ copies of $\End(V_l)$,
$\{f_{l,ij}\}_{l\in\BZ/n\BZ,\ 1\le i\le d_l,\ 1\le j\le d_{l+1}}$ be the coordinates on $\Hom(V_l,V_{l+1})$,
$\{p_{l,i}\}_{l\in\BZ/n\BZ,\ 1\le i\le d_l}$ be the coordinates on $V_l$, $\{q_{l,i}\}_{l\in\BZ/n\BZ,\ 1\le i\le d_l}$ be the coordinates on $V_l^*$. Then the Lie--Poisson bracket on $\fa_{\ul{d}}^*$ reads
\begin{equation}
[e_{l,i_1j_1},e_{k,i_2j_2}]=\delta_{kl}(\delta_{i_2j_1}e_{l,i_1j_2}-\delta_{i_1j_2}e_{l,i_2j_1}),
\end{equation}

\begin{equation}
[e_{l,i_1j_1}',e_{k,i_2j_2}']=\delta_{kl}(\delta_{i_2j_1}e_{l,i_1j_2}'-\delta_{i_1j_2}e_{l,i_2j_1}'),
\end{equation}

\begin{equation}
[e_{l,i_1j_1},e_{k,i_2j_2}']=[e_{l,ij},p_{k,m}]=[e_{l,ij}',q_{k,m}]=0,
\end{equation}

\begin{equation}
[e_{l,ij},q_{k,m}]=\delta_{lk}\delta_{jm}q_{k,i},\quad
[e_{l,ij}',p_{k,m}]=-\delta_{lk}\delta_{mi}p_{k,j}.
\end{equation}

\begin{equation}
[q_{k,i},p_{l,j}]=\delta_{l,k+1}f_{k,ij}.
\end{equation}

Consider the subvariety $S_{\ul{d}}\subset\fa_{\ul{d}}^*$ defined by the following equations:
\begin{equation}\label{poisson-ideal}
B_lA_l+A_{l+1}'B_l+p_{l+1}q_l=0,\quad l\in\BZ/n\BZ.
\end{equation}

\begin{prop} $S_{\ul{d}}$ is $\ad^*(\fa_{\ul{d}})$-invariant (equivalently, the ideal generated by (\ref{poisson-ideal}) is a Poisson ideal).
\end{prop}

\begin{proof}
Straightforward.
\end{proof}

Let $\fgl(V_l)_{\diag}$ be the diagonal $\fgl(V_l)$ inside $\fgl(V_l)\oplus\fgl(V_l)\subset\fa_{\ul{d}}$ and $\pi:\fa_{\ul{d}}^*\to\fgl(V_l)^*_{\diag}$ be the projection. Then the Drinfeld Zastava space $\fZ_{\ul{d}}$ is identified with the Hamiltonian reduction $S_{\ul{d}}//\bigoplus\limits_{l\in\BZ/n\BZ}\fgl(V_l)_{\diag}=\pi^{-1}(0)/GL(V_l)_{\diag}$. This provides a natural Poisson bracket on $\fZ_{\ul{d}}$.

We consider the following polynomial invariants $$a_{l,r}:=\Tr A_l^r=\sum\limits_{i_1,\ldots,i_r} e_{l,i_1i_2}e_{l,i_2i_3}\ldots e_{l,i_ri_1},\ r=1,\ldots, d_l,\ l\in\BZ/n\BZ;$$ $$b_{l,s}:=\langle q_l, A_l^sp_l\rangle=\sum\limits_{i_1,\ldots,i_{s+1}} p_{l,i_1}e_{l,i_1i_2}e_{l,i_2i_3}\ldots e_{l,i_si_{s+1}}q_{l,i_{s+1}},\ s=0,\ldots,d_l-1,\ l\in\BZ/n\BZ.$$

We also introduce the following elements:

\begin{equation}\label{defn-poisson-bkl}
b_{kl;s_k,\ldots,s_l}:=\langle q_l, A_l^{s_l}\prod\limits_{m=k}^{l-1}B_mA_m^{s_m}p_k\rangle,\quad k\le l,\ s_m\in\BZ_{\ge0}.
\end{equation}

\begin{lem} Let $1\le k < l+1\le n-1$. Then $\{b_{kl;s_k,\ldots,s_l},b_{l+1,r}\}=b_{k,l+1;s_k,\ldots,s_l,r}$.
\end{lem}

\begin{proof}
Straightforward.
\end{proof}

\begin{prop}\label{poisson-generators} For $d_0=0$ the coordinate ring of $\fZ_{\ul{d}}$ is generated (as a Poisson algebra) by $a_{l,r},b_{l,s}$ with $l\in\BZ/n\BZ,\ r=1,\ldots, d_l,\ s=0,\ldots,d_l-1$.
\end{prop}

\begin{proof}
According to classical invariant theory, the coordinate ring of $\fZ_{\ul{d}}$ is generated by the elements $a_{l,r}$ and $b_{kl;s_k,\ldots,s_l}$.
Due to the relation~(\ref{poisson-ideal}), one can express $b_{kl;s_k,\ldots,s_l}$ via the sum of products of $a_{l,r}$'s, $b_{l,s}$'s and $b_{kl;s_k,0,\ldots,0}$'s. Now it remains to note that $b_{kl;s_k,0,\ldots,0}=\{\{\ldots\{b_{k,s_k}b_{k+1,0}\}\ldots,b_{l-1,0}\},b_{l,0}\}$.
\end{proof}

\begin{rem}
This is not the case when all of the $d_l$'s are nonzero. There are additional generators of the form $\Tr (\prod\limits_{m=0}^{l-1}B_m)^r$ in general.
\end{rem}

Let $c_{kl}$ be the coefficients of the Cartan matrix (i.e. $c_{kk}=2,\ c_{k,k+1}=c_{k+1,k}=-1,\ c_{kl}=0$ for $|k-l|>1$.)

\begin{prop}\label{prop-poisson-relations-general} For $n\ge3$ the following holds:
\begin{eqnarray}
\{a_{k,r},a_{l,s}\}=0;\\
\{a_{k,r},b_{l,s}\}=\delta_{kl}rb_{l,r+s-1};\\
\{b_{k,r+1},b_{l,s}\}-\{b_{k,r},b_{l,s+1}\}=c_{kl}b_{k,r}b_{l,s};\\
\{b_{k,r_2},\{b_{k,r_1},b_{l,s}\}\}+\{b_{k,r_1},\{b_{k,r_2},b_{l,s}\}\}=0\quad\text{for}\ |k-l|=1.
\end{eqnarray}
\end{prop}

\begin{proof}
The first two relations, as well as third one for $k=l$, follow immediately from Proposition~\ref{prop-poisson-relations}. Let us prove the third relation for $l=k+1$. We have
\begin{multline*}
\{b_{k,r+1},b_{l,s}\}-\{b_{k,r},b_{l,s+1}\}=\langle q_l, A_l^{s}B_kA_k^{r+1}p_k\rangle-\langle q_l, A_l^{s+1}B_kA_k^{r}p_k\rangle=\\=(-1)^s\langle q_l, A_l'^{s}B_kA_k^{r+1}p_k\rangle-(-1)^{s+1}\langle q_l, A_l'^{s+1}B_kA_k^{r}p_k\rangle=\\
=(-1)^{s+1}\langle q_l, A_l'^{s}p_{l}q_kA_k^{r}p_k\rangle=-\langle q_l, A_l^{s}p_{l}\rangle\langle q_k,A_k^{r}p_k\rangle=-b_{k,r}b_{l,s}.
\end{multline*}

Now let us prove the last relation. Assume that $l-k=1,\ r_1\le r_2$. We have
\begin{multline*}
\{b_{k,r_2},\{b_{k,r_1},b_{l,s}\}\}+\{b_{k,r_1},\{b_{k,r_2},b_{l,s}\}\}=\{b_{k,r_2},b_{kl;r_1,s}\}+\{b_{k,r_1},b_{kl;r_2,s}\}=\\
=\sum\limits_{t=r_1}^{r_2-1}b_{k,t}b_{kl;r_1+r_2-t-1,s}-\sum\limits_{t=r_1}^{r_2-1}b_{k,t}b_{kl;r_1+r_2-t-1,s}=0.
\end{multline*}

\end{proof}

Let $x_{l,i}, y_{l,i}$, where $l\in\BZ/n\BZ,\ 1\le i\le d_l$, be the following \'etale coordinates on $\fZ_d$:

\begin{eqnarray}
a_{l,r}=\sum\limits_{i=1}^{d_l} x_{l,i}^r,\quad y_{l,i}=\sum\limits_{r=0}^{d_l-1} (-1)^{r}\sigma_r(x_{l,1},\ldots, \widehat{x_{l,i}},\ldots, x_{l,d_l})b_{l,d_l-1-r},
\end{eqnarray}
where $\sigma_r$ stands for the elementary symmetric function of degree $r$.

\begin{prop}\label{kuzn-general}
We have \begin{eqnarray}\{x_{k,i},x_{l,j}\}=0,\\ \{x_{k,i},y_{l,j}\}=\delta_{kl}\delta_{ij}y_j,\\ \{y_{k,i},y_{l,j}\}=\frac{(2\delta_{kl}-c_{kl})y_{k,i}y_{l,j}}{x_{k,i}-x_{l,j}}.\end{eqnarray}
\end{prop}

\begin{proof}
The first two relations follow immediately from Proposition~\ref{kuzn-sl2}. Let us prove the last one.

Denote by $\ft_{\ul{d}}$ the Cartan subalgebra of $\fgl(V_{\ul{d}})$. Consider the $\ft_{\ul{d}}$-invariant subspace $S:=\ft_{\ul{d}}^*\oplus \bigoplus\limits_{l\in\BZ/n\BZ} \fn(V_l,V_{l+1})^*\subset\fa_{\ul{d}}^*$. Note that each $GL(V_{\ul{d}})$-invariant function on $\fa_{\ul{d}}^*$ is uniquely determined by its restriction to this subspace. We have $$
y_{k,i}|_S=\prod\limits_{m\ne i}(x_{k,i}-x_{k,m})p_{k,i}q_{k,i}\ \text{and}\ dy_{k,i}|_S=\prod\limits_{m\ne i}(x_{k,i}-x_{k,m})p_{k,i}dq_{k,i}+\omega,
$$ where $\omega$ has the form $\sum\limits_{i,j} F_{ij}(x_{k,m},p_{k,m},q_{k,m})de_{ij}+\sum\limits_{i} F_{i}(x_{k,m},p_{k,m},q_{k,m})dp_{k,i}$.
Hence for $l=k+1$ we have (since $\ad(\omega)$ centralizes $y_{l,j}$)
\begin{multline*}\{y_{k,i},y_{l,j}\}|_S=\ad(dy_{k,i})(y_{l,j})|_S=\prod\limits_{m\ne i}(x_{k,i}-x_{k,m})p_{k,i}\ad(q_{k,i})(y_{l,j})|_S=\\
=\prod\limits_{m\ne i}(x_{k,i}-x_{k,m})p_{k,i}f_{k,ij}\prod\limits_{m\ne j}(x_{l,j}-x_{l,m})q_{l,j}|_S.
\end{multline*}
According to~(\ref{poisson-ideal}), the latter is
$$\prod\limits_{m\ne i}(x_{k,i}-x_{k,m})p_{k,i}\frac{q_{k,i}p_{l,j}}{x_{k,i}-x_{l,j}}\prod\limits_{m\ne j}(x_{l,j}-x_{l,m})q_{l,j}|_S=\frac{y_{k,i}y_{l,j}}{x_{k,i}-x_{l,j}}|_S.
$$
\end{proof}

\begin{cor} $\{\cdot,\cdot\}=\{\cdot,\cdot\}_K$ on $\fZ_{\ul{d}}$.
\end{cor}

\subsection{Quantum Hamiltonian reduction. $\fsl_2$ case.}

The natural quantization of the coordinate ring of the space $\fZ_d$ is the \emph{quantum Hamiltonian reduction} $\YO_d:=\left(U(\fa)/U(\fa)\fgl(V)_{\diag}\right)^{\fgl(V)_{\diag}}$.

The algebra $\left(U(\fa)/U(\fa)\fgl(V)_{\diag}\right)^{\fgl(V)_{\diag}}$ is generated by the following elements: $$a_r:=\sum\limits_{i_1,\ldots,i_r} e_{i_1i_2}e_{i_2i_3}\ldots e_{i_ri_1},\ r=1,\ldots, d;$$ $$b_s:=\sum\limits_{i_1,\ldots,i_{s+1}} p_{i_1}e_{i_1i_2}e_{i_2i_3}\ldots e_{i_si_{s+1}}q_{i_{s+1}},\ s=0,\ldots,d-1.$$
We also set $a_0:=d$.

Introduce the following notation:

$$
e_{ij}^{(r)}:=\sum\limits_{i_1,\ldots,i_{r-1}} e_{ii_1}e_{i_1i_2}e_{i_2i_3}\ldots e_{i_{r-1}j}.
$$
We also set $e_{ij}^{(0)}=\delta_{ij}$.

We will use the following relations:
\begin{lem}
\begin{equation} \sum\limits_{i}e_{ki}^{(r)}e_{ij}^{(s)}=e_{kj}^{(r+s)};
\end{equation}
\begin{equation} [ e_{kl},e_{ij}^{(r)}] =\delta_{il}e_{kj}^{(r)}-\delta_{kj}e_{il}^{(r)};
\end{equation}
\begin{equation} [ e_{kl},\sum\limits_{j=1}^de_{ij}^{(r)}q_j] =\delta_{il}\sum\limits_{j=1}^de_{kj}^{(r)}q_j;
\end{equation}
\begin{equation} [ e_{kl}^{(s)},\sum\limits_{j=1}^de_{ij}^{(r)}q_j] =\sum\limits_{t=1}^{s}\sum\limits_{j=1}^de_{kj}^{(r+t-1)}q_je_{il}^{(s-t)}.
\end{equation}
\end{lem}

\begin{proof}
Straightforward.
\end{proof}

\begin{prop}\label{prop-quantum-relations}
\begin{equation}
[a_r,a_s]=0;\end{equation}
\begin{equation}
[a_1,b_s]=b_{s};\end{equation}
\begin{equation}
[a_{r+1},b_s]-[a_r,b_{s+1}]=b_{r+s}-\sum\limits_{t=0}^{r-1} b_{r+s-t-1}a_t;\end{equation}
\begin{equation}
[b_{r+1},b_s]-[b_r,b_{s+1}]=b_rb_s+b_sb_r.
\end{equation}
\end{prop}

\begin{proof} The first two relations are obvious. Let us prove the third one. We have
\begin{multline*}
[a_{r+1},b_s]-[a_r,b_{s+1}]=[\sum\limits_{k} e_{kk}^{(r+1)},\sum\limits_{ij} p_{i}e_{ij}^{(s)}q_{j}]-[\sum\limits_{k} e_{kk}^{(r)},\sum\limits_{ij} p_{i}e_{ij}^{(s+1)}q_{j}]=\\
=\sum\limits_{t=1}^{r+1}\sum\limits_{k,i,j}p_{i}e_{kj}^{(t+s-1)}q_je_{ik}^{(r+1-t)}-\sum\limits_{t=1}^{r}\sum\limits_{k,i,j}p_{i}e_{kj}^{(t+s-1)}q_je_{ik}^{(r-t)}
=\sum\limits_{k,i,j}p_{i}e_{kj}^{(s)}q_je_{ik}^{(r)}.
\end{multline*}
Now it suffices to check that \begin{equation*}
\sum\limits_{k,i,j}p_{i}e_{kj}^{(s)}q_je_{ik}^{(r)}=b_{r+s}-\sum\limits_{t=0}^{r-1} b_{r+s-t-1}a_t.
\end{equation*}
This is clear for $r=1$. Assume this for $r=R$ and prove for $r=R+1$:
\begin{multline*}
\sum\limits_{k,i,j}p_{i}e_{kj}^{(s)}q_je_{ik}^{(R+1)}=\sum\limits_{k,l,i,j}p_{i}e_{kj}^{(s)}q_je_{il}^{(R)}e_{lk}=\\
=\sum\limits_{k,l,i,j}p_{i}e_{lk}e_{kj}^{(s)}q_je_{il}^{(R)}-d\sum\limits_{l,i,j}p_{i}e_{lj}^{(s)}q_je_{il}^{(R)}+d\sum\limits_{k,i,j}p_{i}e_{kj}^{(s)}q_je_{ik}^{(R)}-\sum\limits_{k,l,i,j}p_{i}e_{kj}^{(s)}q_j\delta_{ik}e_{ll}^{(R)}=\\
=-b_sa_R+b_{R+s+1}-\sum\limits_{t=0}^{R-1} b_{R+s-t}a_t=b_{R+s+1}-\sum\limits_{t=0}^{R} b_{R+s-t}a_t.
\end{multline*}

Finally, let us prove the last relation.
We will use the following
\begin{lem}\label{lem-comp-b}
$\sum\limits_{t=0}^{s-1}\sum\limits_{m,j}e_{ij}^{(s-t-1)}q_je_{lm}^{(t)}q_m=
\sum\limits_{t=0}^{s-1}\sum\limits_{m,j}e_{lm}^{(t)}q_me_{ij}^{(s-t-1)}q_j$.
\end{lem}

\begin{proof}
Induction on $s$.
\begin{multline*}\sum\limits_{t=0}^{s-1}\sum\limits_{m,j}e_{ij}^{(s-t-1)}q_je_{lm}^{(t)}q_m=\\
=\sum\limits_{t=0}^{s-1}\sum\limits_{m,j}e_{lm}^{(t)}q_me_{ij}^{(s-t-1)}q_j+\\
+\sum\limits_{t=0}^{s-1}\sum\limits_{u=1}^{s-t-1}\sum\limits_{m,j,k}e_{lm}^{(t)}e_{ik}^{(u-1)}q_ke_{mj}^{(s-t-1-u)}q_j
-\sum\limits_{t=0}^{s-1}\sum\limits_{u=1}^{t}\sum\limits_{m,j,k}e_{lk}^{(u-1)}e_{kj}^{(s-t-1)}q_je_{im}^{(s-t-1-u)}q_m=\\
=\sum\limits_{t=0}^{s-1}\sum\limits_{m,j}e_{lm}^{(t)}q_me_{ij}^{(s-t-1)}q_j
+\sum\limits_{t=0}^{s-1}\sum\limits_{u=1}^{s-t-1}\sum\limits_{m,j,k}e_{lm}^{(t)}(e_{ik}^{(u-1)}q_ke_{mj}^{(s-t-1-u)}q_j-e_{mj}^{(s-t-1-u)}q_je_{ik}^{(u-1)}q_k)=\\
=\sum\limits_{t=0}^{s-1}\sum\limits_{m,j}e_{lm}^{(t)}q_me_{ij}^{(s-t-1)}q_j
\end{multline*}
\end{proof}

Now we are ready to check the relation on $b_r$:

\begin{multline*}
[b_{r+1},b_s]-[b_r,b_{s+1}]=[\sum\limits_{k,l} p_{k}e_{kl}^{(r+1)}q_{l},\sum\limits_{i,j} p_{i}e_{ij}^{(s)}q_{j}]-[\sum\limits_{k,l} p_{k}e_{kl}^{(r)}q_{l},\sum\limits_{i,j} p_{i}e_{ij}^{(s+1)}q_{j}]=\\
=\sum\limits_{t=1}^{r+1}\sum\limits_{k,l,i,j}p_kp_{i}e_{kj}^{(t+s-1)}q_je_{il}^{(r+1-t)}q_l-
\sum\limits_{t=0}^{s-1}\sum\limits_{k,l,m,i,j}p_kp_{i}e_{kl}^{(r+1)}e_{im}^{(t)}q_me_{lj}^{(s-t-1)}q_j-\\
-\sum\limits_{t=1}^{r}\sum\limits_{k,l,i,j}p_kp_{i}e_{kj}^{(t+s-1)}q_je_{il}^{(r-t)}q_l+
\sum\limits_{t=0}^{s}\sum\limits_{k,l,m,i,j}p_kp_{i}e_{kl}^{(r)}e_{im}^{(t)}q_me_{lj}^{(s-t)}q_j=\\
=\sum\limits_{k,l,i,j}p_kp_{i}e_{kj}^{(s)}q_je_{il}^{(r)}q_l
-\sum\limits_{t=0}^{s-1}\sum\limits_{k,l,m,i,j}p_kp_{i}e_{kl}^{(r+1)}e_{im}^{(t)}q_me_{lj}^{(s-t-1)}q_j
+\sum\limits_{t=0}^{s}\sum\limits_{k,l,m,i,j}p_kp_{i}e_{kl}^{(r)}e_{im}^{(t)}q_me_{lj}^{(s-t)}q_j=\\
=b_sb_r
-\sum\limits_{t=0}^{s-1}\sum\limits_{k,l,m,i,j}p_kp_{i}e_{kl}^{(r+1)}e_{im}^{(t)}q_me_{lj}^{(s-t-1)}q_j
+\sum\limits_{t=0}^{s-1}\sum\limits_{k,l,m,i,j}p_kp_{i}e_{kl}^{(r+1)}e_{im}^{(t)}q_me_{lj}^{(s-t-1)}q_j+\\
+\sum\limits_{k,l,m,i,j}p_kp_{i}e_{kl}^{(r)}e_{im}^{(s)}q_m\delta_{lj}q_j
-\sum\limits_{t=0}^{s-1}\sum\limits_{k,l,m,i,j}p_kp_{i}e_{kl}^{(r)}e_{lm}^{(t)}q_me_{ij}^{(s-t-1)}q_j=\\
=b_sb_r+\sum\limits_{k,l,m,i}p_kp_{i}e_{kl}^{(r)}e_{im}^{(s)}q_mq_l
-\sum\limits_{t=0}^{s-1}\sum\limits_{k,l,m,i,j}p_kp_{i}e_{kl}^{(r)}e_{lm}^{(t)}q_me_{ij}^{(s-t-1)}q_j=\\
=b_sb_r+b_rb_s+\sum\limits_{t=0}^{s-1}\sum\limits_{k,l,m,i,j}p_kp_{i}e_{kl}^{(r)}e_{ij}^{(s-t-1)}q_je_{lm}^{(t)}q_m
-\sum\limits_{t=0}^{s-1}\sum\limits_{k,l,m,i,j}p_kp_{i}e_{kl}^{(r)}e_{lm}^{(t)}q_me_{ij}^{(s-t-1)}q_j.
\end{multline*}

According to Lemma~\ref{lem-comp-b}, the latter is $b_sb_r+b_rb_s$.

\end{proof}

\subsection{Quantum Hamiltonian reduction. General case.}

Consider the subspace $R$ in the universal enveloping algebra $U(\fa_{\ul{d}})$ consisting of the quadratic elements

\begin{equation}\label{2sided-ideal}
\sum\limits_{m=1}^{d_l} e_{l,mj}f_{l,im}+\sum\limits_{m=1}^{d_{l+1}}f_{l,mj}e'_{l+1,im}+\frac{1}{2}(p_{l+1,i}q_{l,j}+q_{l,j}p_{l+1,i}),\quad l\in\BZ/n\BZ,\ i=1,\ldots,d_{l+1},\ j=1,\ldots,d_{l}.
\end{equation}

\begin{prop} We have $[\fa_{\ul{d}},R]\subset R$ (equivalently, $U(\fa_{\ul{d}})R$ is a two-sided ideal in $U(\fa_{\ul{d}})$).
\end{prop}

\begin{proof}
Straightforward.
\end{proof}

The natural quantization of the coordinate ring of the space $\fZ_{\ul{d}}$ is the \emph{quantum Hamiltonian reduction} $\YO_{\ul{d}}:=\left(U(\fa_{\ul{d}})/U(\fa_{\ul{d}})(R+\fgl(V_{\ul{d}})_{\diag})\right)^{\fgl(V_{\ul{d}})_{\diag}}$. The ring $\YO_{\ul{d}}$ has a natural filtration coming from the PBW filtration on $U(\fa_{\ul{d}})$.

\begin{prop}[PBW property]\label{PBW} We have $\gr\ \YO_{\ul{d}}=\BC[\fZ_{\ul{d}}]$.
\end{prop}

\begin{proof}
Clearly, the graded vector space $\gr\ \YO_{\ul{d}}$ is not bigger than $\BC[\fZ_{\ul{d}}]$ (i.e. the dimension of each component of $\gr\ \YO_{\ul{d}}$ is not greater than that of the corresponding component of $\BC[\fZ_{\ul{d}}]$). Let us show that $\gr\ \YO_{\ul{d}}$ is not smaller than $\BC[\fZ_{\ul{d}}]$.

Let $\overline{R}=\gr\ R\subset S(\fa_{\ul{d}})$ be the space of quadratic relations~(\ref{poisson-ideal}). This space of quadratic relations together with $\fgl(V_{\ul{d}})_{\diag}\subset S(\fa_{\ul{d}})$ defines the coordinate ring $\BC[\sM_{\ul{d}}]=S(\fa_{\ul{d}})/S(\fa_{\ul{d}})(\overline{R}+\fgl(V_{\ul{d}})_{\diag})$. Since $\sM_{\ul{d}}$ is a complete intersection, the Koszul complex $S(\fa_{\ul{d}})\otimes\Lambda^{\bullet}(\overline{R}+\fgl(V_{\ul{d}})_{\diag})$ is a resolution of the $S(\fa_{\ul{d}})$-module $S(\fa_{\ul{d}})/S(\fa_{\ul{d}})(\overline{R}+\fgl(V_{\ul{d}})_{\diag})$.

Since $[R+\fgl(V_{\ul{d}})_{\diag},R+\fgl(V_{\ul{d}})_{\diag}]\subset (\fa_{\ul{d}}+\BC)R+\fgl(V_{\ul{d}})_{\diag}$, there is a resolution of the left $U(\fa_{\ul{d}})$-module $U(\fa_{\ul{d}})/U(\fa_{\ul{d}})(R+\fgl(V_{\ul{d}})_{\diag})$, beginning with
\begin{multline*}
U(\fa_{\ul{d}})\otimes\Lambda^{2}(R+\fgl(V_{\ul{d}})_{\diag})\to U(\fa_{\ul{d}})\otimes\Lambda^{1}(R+\fgl(V_{\ul{d}})_{\diag})\to\\\to U(\fa_{\ul{d}})\to U(\fa_{\ul{d}})/U(\fa_{\ul{d}})(R+\fgl(V_{\ul{d}})_{\diag})\to0,
\end{multline*}
which deforms the corresponding segment of the Koszul resolution of $S(\fa_{\ul{d}})/S(\fa_{\ul{d}})(\overline{R}+\fgl(V_{\ul{d}})_{\diag})$. The image of $U(\fa_{\ul{d}})\otimes\Lambda^{2}(R+\fgl(V_{\ul{d}})_{\diag})$ in $U(\fa_{\ul{d}})\otimes\Lambda^{1}(R+\fgl(V_{\ul{d}})_{\diag})$ is not smaller (as a filtered vector space) than the image of $S(\fa_{\ul{d}})\otimes\Lambda^{2}(\overline{R}+\fgl(V_{\ul{d}})_{\diag})$ in $S(\fa_{\ul{d}})\otimes\Lambda^{1}(\overline{R}+\fgl(V_{\ul{d}})_{\diag})$, since the differential deforms the Koszul differential.
Hence $\gr\ U(\fa_{\ul{d}})/U(\fa_{\ul{d}})(R+\fgl(V_{\ul{d}})_{\diag})$ is not smaller than $S(\fa_{\ul{d}})/S(\fa_{\ul{d}})(\overline{R}+\fgl(V_{\ul{d}})_{\diag})$. Since $\fgl(V_{\ul{d}})_{\diag}$ is reductive, the same holds for $\fgl(V_{\ul{d}})_{\diag}$-invariants. Hence $\gr\ \YO_{\ul{d}}=\BC[\fZ_{\ul{d}}]$.
\end{proof}

We consider the following elements of $\YO_{\ul{d}}$:
$$a_{l,r}:=\sum\limits_{i_1,\ldots,i_r} e_{l,i_1i_2}e_{l,i_2i_3}\ldots e_{l,i_ri_1},\
r=1,2,\ldots,\ l\in\BZ/n\BZ;$$ $$b_{l,s}:=\sum\limits_{i_1,\ldots,i_{s+1}} p_{l,i_1}e_{l,i_1i_2}e_{l,i_2i_3}\ldots e_{l,i_si_{s+1}}q_{l,i_{s+1}},\ s=0,1,\ldots,\ l\in\BZ/n\BZ.$$

\begin{prop} For $d_0=0$ the algebra $\YO_{\ul{d}}$ is generated by $a_{l,r},b_{l,s}$ with $l\in\BZ/n\BZ,\ r=1,\ldots, d_l,\ s=0,\ldots,d_l-1$.
\end{prop}

\begin{proof}
This follows from Propositions~\ref{poisson-generators}~and~\ref{PBW}.
\end{proof}

Introduce the following generating series
\begin{equation}
a_l(u):=1-d_lu^{-1}-\sum\limits_{r=1}^\infty a_{l,r}u^{-r-1}, \quad b_l(u):=\sum\limits_{s=0}^\infty b_{l,s}u^{-s-1}
\end{equation}

We also consider the elements $$b'_{l,s}:=(-1)^s\sum\limits_{i_1,\ldots,i_{s+1}} p_{l,i_1}e'_{l,i_1i_2}e'_{l,i_2i_3}\ldots e'_{l,i_si_{s+1}}q_{l,i_{s+1}},\ s=0,\ldots,d_l-1,\ l\in\BZ/n\BZ,$$
and the corresponding generating series $b'_l(u):=\sum\limits_{s=0}^\infty b'_{l,s}u^{-s-1}$.

\begin{lem}\label{bb'}
$b'_l(u)=b_l(u+d_l)$.
\end{lem}

\begin{proof} Straightforward.
\end{proof}

We also introduce the following elements:

\begin{multline}
b_{kl;s_k,\ldots,s_l}:=\\
=\sum\limits_{i^{l}_{1},\ldots,i^{l}_{s_l+1}}\ldots\sum\limits_{i^{k}_{1},\ldots,i^{k}_{s_k+1}}p_{l,i^{l}_1}e_{l,i^{l}_1i^{l}_2}e_{l,i^{l}_2i^{l}_3}\ldots e_{l,i^{l}_{s_l}i^{l}_{s_l+1}}f_{l-1,i^{l}_{s_l+1}i^{l-1}_{1}}e_{l-1,i^{l-1}_{1}i^{l-1}_{2}}\ldots f_{k,i^{k+1}_{s_{k+1}+1}i^{k}_{1}}e_{k,i^{k}_{1}i^{k}_{2}}\ldots q_{k,i^{k}_{s_k+1}},\\ k\le l,\ s_m\in\BZ_{\ge0}.
\end{multline}

\begin{lem} Let $1\le k < l+1\le n-1$. Then $[b_{kl;s_k,\ldots,s_l},b_{l+1,r}]=b_{k,l+1;s_k,\ldots,s_l,r}$.
\end{lem}

\begin{proof}
Straightforward.
\end{proof}

\begin{prop}\label{prop-quantum-relations-general} For $n\ge3$ the following holds:
\begin{equation}
[a_{k,r},a_{l,s}]=0;
\end{equation}
\begin{equation}
[a_{k,1},b_{l,s}]=\delta_{kl}b_{l,s};\end{equation}
\begin{equation}
[a_{k,r+1},b_{l,s}]-[a_{k,r},b_{l,s+1}]=\delta_{kl}(b_{l,r+s}-\sum\limits_{t=0}^{r-1} b_{l,r+s-t-1}a_{k,t});\end{equation}
\begin{equation}
[b_{k,r+1},b_{l,s}]-[b_{k,r},b_{l,s+1}]=b_{k,r}b_{l,s}+b_{l,s}b_{k,r}\quad\text{for}\ l=k;
\end{equation}
\begin{equation}
[b_{k,r+1},b'_{l,s}]-[b_{k,r},b'_{l,s+1}]=-\frac{1}{2}(b_{k,r}b'_{l,s}+b'_{l,s}b_{k,r})\quad\text{for}\ l=k+1;
\end{equation}
\begin{equation}
[b_{k,r_2},[b_{k,r_1},b_{l,s}]]+[b_{k,r_1},[b_{k,r_2},b_{l,s}]]=0\quad\text{for}\ |k-l|=1.
\end{equation}
\end{prop}

\begin{proof}
The first four relations follow immediately from Proposition~\ref{prop-quantum-relations}. Assume that $l-k=1$. Arguing in the same way as in Proposition~\ref{prop-poisson-relations-general}, we have $[b_{k,r+1},b'_{l,s}]-[b_{k,r},b'_{l,s+1}]=-\frac{1}{2}(b_{k,r}b'_{l,s}+b'_{l,s}b_{k,r})$.

Now let us prove the last relation. Assume that $l-k=1,\ r_1\le r_2$.
\begin{multline*}
[b_{k,r_2},[b_{k,r_1},b_{l,s}]]+[b_{k,r_1},[b_{k,r_2},b_{l,s}]]=[b_{k,r_2},b_{kl;r_1,s}]+[b_{k,r_1},b_{kl;r_2,s}]=\\
=\sum\limits_{t=r_1}^{r_2-1}b_{k,t}b_{kl;r_1+r_2-t-1,s}-\sum\limits_{t=r_1}^{r_2-1}b_{k,t}b_{kl;r_1+r_2-t-1,s}=0.
\end{multline*}

\end{proof}

\subsection{Deformation of affine Zastava spaces}

The affine Zastava space admits the following nontrivial deformation. Fix a character $\ul{\mu}=\sum\limits_{l\in\BZ/n\BZ}\mu_l\Tr_{V_l}$ of the Lie algebra $\bigoplus\limits_{l\in\BZ/n\BZ}\fgl(V_l)_{\diag}$ and consider the Hamiltonian reduction of $S_{\ul{d}}$ at this character $\fZ_{\ul{d}}^{\ul{\mu}}:=\pi^{-1}(\ul{\mu})/\prod GL(V_l)_{\diag}$.

The following Poisson automorphisms of $\fa_{\ul{d}}^*$ preserve $S_{\ul{d}}$:
\begin{equation}
\varphi_{\ul{\nu},\ul{\nu'}}:A_l\mapsto A_l+\nu_lE,\ A'_l\mapsto A'_l+\nu'_lE,\ B_l\mapsto B_l, p_l\mapsto p_l,\ q_l\mapsto q_l
\end{equation}
with $\nu_l+\nu'_{l+1}=0$. We have $\varphi_{\ul{\nu},\ul{\nu'}}(\ul{\mu})=\sum\limits_{l\in\BZ/n\BZ}(\mu_l+d_l(\nu_l-\nu_{l-1}))\Tr_{V_l}$. Hence the isomorphism class of $\fZ_{\ul{d}}^{\ul{\mu}}$ depends only on $|\ul{\mu}|:=\sum\limits_{l=1}^n\mu_l$.

\begin{rem} \emph{For $n=1$ this is precisely the Calogero--Moser deformation of the Hilbert scheme.}
\end{rem}

As in the non-deformed situation, we consider the following polynomial invariants $$a_{l,r}:=\Tr (A_l-\frac{\mu_l}{d_l}E)^r,\ r=1,\ldots, d_l,\ l\in\BZ/n\BZ;$$ $$b_{l,s}:=\langle q_l, (A_l-\frac{\mu_l}{d_l}E)^sp_l\rangle,\ s=0,\ldots,d_l-1,\ l\in\BZ/n\BZ.$$

We also introduce the following elements:

\begin{equation}
b_{kl;s_k,\ldots,s_l}:=\langle q_l, (A_l-\frac{\mu_l}{d_l}E)^{s_l}\prod\limits_{m=k}^{l-1}B_m(A_m-\frac{\mu_m}{d_m}E)^{s_m}p_k\rangle,\quad k\le l,\ s_m\in\BZ_{\ge0}.
\end{equation}

The same can be done on the quantum level. We obtain a quantization of deformed affine Zastava spaces $\YO_{\ul{d}}^{\ul{\mu}}:=\left(U(\fa_{\ul{d}})/U(\fa_{\ul{d}})(R+\fgl(V_{\ul{d}})_{\diag}-\ul{\mu}(\fgl(V_{\ul{d}})_{\diag})\right)^{\fgl(V_{\ul{d}})_{\diag}}$.
This algebra also depends only on $|\ul{\mu}|$. The PBW property also holds for $\YO_{\ul{d}}^{\ul{\mu}}$: one has $\gr \YO_{\ul{d}}^{\ul{\mu}}=\gr \BC[\fZ_{\ul{d}}^{\ul{\mu}}]=\BC[\fZ_{\ul{d}}]$. The proof is the same as for Proposition~\ref{PBW}.

\begin{prop}\label{poisson-generators-affine} For $|\ul{\mu}|\ne0$, the coordinate ring of $\fZ_{\ul{d}}^{\ul{\mu}}$ is generated (as a Poisson algebra) by $a_{l,r},b_{l,s}$ with $l\in\BZ/n\BZ,\ r=1,\ldots, d_l,\ s=0,\ldots,d_l-1$.
\end{prop}

\begin{proof}
According to classical invariant theory, the coordinate ring of $\fZ_{\ul{d}}$ is generated by the elements $a_{l,r}$, $b_{kl;s_k,\ldots,s_l}$ and $C_{r;s_0,\ldots,s_{rn-1}}:=\Tr (\prod\limits_{m=0}^{rn-1}A^{s_m}B_m)$ for $r=1,2,\ldots$.
Due to the relation~(\ref{poisson-ideal}), one can express $b_{kl;s_k,\ldots,s_l}$ via the sum of products of $a_{l,r}$'s, $b_{l,s}$'s and $b_{kl;s_k,0,\ldots,0}$'s. Analogously, one can express $C_{r;s_0,\ldots,s_{rn-1}}$ via the sum of products of $C_{r;s_0,0\ldots,0}$, $a_{l,r}$'s, $b_{l,s}$'s and $b_{kl;s_k,0,\ldots,0}$'s.

Consider the filtration on $\BC[\fZ_{\ul{d}}]$ by the degree in $f_{l,ij}$, the coefficients of the $B_l$'s. With respect to this filtration, we have
$$b_{kl;s_k,0,\ldots,0}=(s_k+1)\{a_{k,s_k+1}, b_{kl;0,0,\ldots,0}\}+\text{lower terms}$$
and
$$C_{r;s_0,0\ldots,0}=(s_0+1)\{a_{0,s_0+1}, C_{r;0,\ldots,0}\}+\text{lower terms}.$$
Hence it is sufficient to show that $b_{kl;0,0,\ldots,0}$ and $C_{r;0,\ldots,0}$ can be expressed via $a_{l,r},b_{l,s}$.

For $l-k<n-2$ we have $b_{kl;0,0,\ldots,0}=\{\{\ldots\{b_{k,0}b_{k+1,0}\}\ldots,b_{l-1,0}\},b_{l,0}\}$.

We have
\begin{equation*}
\{b_{0,0}, b_{1,n-1;0,0,\ldots,0}\}=b_{0,n-1;0,0,\ldots,0}-b_{1,n;0,0,\ldots,0}.
\end{equation*}
Hence each linear combination $\sum m_kb_{k,k+n-1;0,0,\ldots,0}$ with $\sum m_k=0$ is expressed via $a_{l,r},b_{l,s}$.

On the other hand, due to the relation~(\ref{poisson-ideal}) $b_{0,n-1;0,0,\ldots,0}-b_{1,n;0,0,\ldots,0}= 2b_{0,n-1;0,0,\ldots,0}+\sum\limits_{k=2}^{n-1}b_{k,k+n-1;0,0,\ldots,0}+|\ul{\mu}|C_{1;0,\ldots,0}$. Hence for each $k$, $b_{k,k+n-1;0,0,\ldots,0}+\frac{|\ul{\mu}|}{n}C_{1;0,\ldots,0}$ is expressed via $a_{l,r},b_{l,s}$.

According to the relation~(\ref{poisson-ideal}), we have
\begin{equation*}
\{a_{1,2}+a'_{2,2}, b_{0,n-1;0,0,\ldots,0}+\frac{|\ul{\mu}|}{n}C_{1;0,\ldots,0}\}=2b_{0,1;0,0}b_{2,n-1;0,0,\ldots,0}+2\frac{|\ul{\mu}|}{n}b_{2,n+1;0,\ldots,0}.
\end{equation*}
Hence $b_{k,k+n-1;0,0,\ldots,0}$ and $C_{1;0,\ldots,0}$ are expressed via $a_{l,r},b_{l,s}$.

Now let us proceed by induction. Suppose that $b_{k,k+N;0,0,\ldots,0}$ and $C_{m;0,\ldots,0}$ are expressed via $a_{l,r},b_{l,s}$ for $N<(r-1)n,\ m<r$. For $(r-1)n-1<l-k<rn-2$ we have $b_{kl;0,0,\ldots,0}=\{\{\ldots\{b_{k,0}b_{k+1,0}\}\ldots,b_{l-1,0}\},b_{l,l+rn-1,0}\}$. The same arguments as above shows that for each $k$, $b_{k,k+rn-1;0,0,\ldots,0}+\frac{r|\ul{\mu}|}{n}C_{r;0,\ldots,0}$ is expressed via $a_{l,r},b_{l,s}$.
According to the relation~(\ref{poisson-ideal}), we have
\begin{equation*}
\{a_{1,2}+a'_{2,2}, b_{0,rn-1;0,0,\ldots,0}+\frac{r|\ul{\mu}|}{n}C_{r;0,\ldots,0}\}=2r\frac{|\ul{\mu}|}{n}b_{2,n+1;0,\ldots,0}+\text{lower terms}.
\end{equation*}
Hence $b_{k,k+rn-1;0,0,\ldots,0}$ and $C_{r;0,\ldots,0}$ are expressed via $a_{l,r},b_{l,s}$.
\end{proof}

\begin{cor}\label{quantum-generators-affine}
For $|\ul{\mu}|\ne0$, $\YO_{\ul{d}}^{\ul{\mu}}$ is generated by $a_{l,r},b_{l,s}$ with $l\in\BZ/n\BZ,\ r=1,\ldots, d_l,\ s=0,\ldots,d_l-1$.
\end{cor}

\section{Yangians}
\subsection{Yangian of $\fsl_n$}
\label{yang fin}
Let $(c_{kl})_{1\leq k,l\leq n-1}$ stand for the Cartan matrix of
$\fsl_n$. The Yangian $Y(\fsl_n)$ is generated
by $\bx_{k,r}^\pm,\bh_{k,r},\ 1\leq k\leq n-1,\ r\in\BN$, with the following
relations (see~\cite{mbook}):

\begin{equation}
\label{11}
[\bh_{k,r},\bh_{l,s}]=0,\ [\bh_{k,0},\bx_{l,s}^\pm]=\pm c_{kl}\bx_{l,s}^\pm,
\end{equation}

\begin{equation}
\label{12}
2[\bh_{k,r+1},\bx_{l,s}^\pm]-2[\bh_{k,r},\bx_{l,s+1}^\pm]=
\pm c_{kl}(\bh_{k,r}\bx_{l,s}^\pm+\bx_{l,s}^\pm\bh_{k,r}),
\end{equation}

\begin{equation}
\label{13}
[\bx^+_{k,r},\bx^-_{l,s}]=\delta_{kl}\bh_{k,r+s},
\end{equation}

\begin{equation}
\label{14}
2[\bx_{k,r+1}^\pm,\bx_{l,s}^\pm]-2[\bx_{k,r}^\pm,\bx_{l,s+1}^\pm]=
\pm c_{kl}(\bx_{k,r}^\pm\bx_{l,s}^\pm+\bx_{l,s}^\pm\bx_{k,r}^\pm),
\end{equation}

\begin{equation}
\label{15}
[\bx_{k,r}^\pm,[\bx_{k,p}^\pm,\bx_{l,s}^\pm]]+
[\bx_{k,p}^\pm,[\bx_{k,r}^\pm,\bx_{l,s}^\pm]]=0,\
k=l\pm1,\ \forall p,r,s\in\BN.
\end{equation}

We will consider the ``Borel subalgebra'' $\YO$ of the Yangian, generated by $\bx_{k,r}^+$ and $\bh_{k,r}$.
For a formal variable $u$ we introduce the generating series
$\bh_k(u):=1+\sum_{r=0}^\infty\bh_{k,r}\hbar^{-r}u^{-r-1};\
\bx_k^+(u):=\sum_{r=0}^\infty\bx_{k,r}^\pm\hbar^{-r}u^{-r-1}$.

We also consider a bigger algebra $\DO\YO$, the ``Borel subalgebra of the Yangian double'', generated by all Fourier components of the series $\bh_k(u):=1+\sum_{r=0}^\infty\bh_{k,r}\hbar^{-r}u^{-r-1};\
\bx_k^+(u):=\sum_{r=-\infty}^\infty\bx_{k,r}^\pm\hbar^{-r}u^{-r-1}$ (i.e. the generating series $\bx_k^+(u)$ are infinite in both positive and negative directions) with the defining relations~(\ref{11},\ref{12},\ref{14},\ref{15}). The algebra $\YO$ is then the subalgebra generated by negative Fourier components of $\bx_k^+(u)$ and $\bh_k(u)$ due to PBW property of the Yangians. We can then rewrite the equations~(\ref{12},\ref{14}) in the following form

\begin{equation}
\label{12'}
\bh_k(u)\bx_l^+(v)\frac{2u-2v-c_{kl}}{2u-2v+ c_{kl}}=\bx_l^+(v)\bh_k(u).
\end{equation}

\begin{equation}
\label{14'}
\bx_k^+(u)\bx_l^+(v)(2u-2v- c_{kl})=(2u-2v+ c_{kl})\bx_l^+(v)\bx_k^+(u).
\end{equation}

The function $\frac{2u-2v-c_{kl}}{2u-2v+ c_{kl}}$ here is understood as a formal power series in $u^{-1},\ v^{-1},\ u^{-1}v$, hence the equation~(\ref{12'}) is well-defined.


Given a sequence $(d_1,\ldots,d_{n-1})$,
we will use a little bit different generators of the Cartan subalgebra of the Yangian,
\begin{equation}
\bA_k(u):=u^{d_k}+A_{k,0}u^{d_k-1}+\ldots+A_{k,r}u^{d_k-r-1}+\ldots,
\end{equation}
obtained as the (unique) solution of the functional equation

\begin{equation}
\bh_k(u)=\bA_k(u+\frac{1}{2})^{-1}\bA_k(u-\frac{1}{2})^{-1}\bA_{k-1}(u)\bA_{k+1}(u)(u+\frac{1}{2})^{d_k}(u-\frac{1}{2})^{d_k}u^{-d_{k-1}}u^{-d_{k+1}},
\end{equation}
where we take $\bA_0(u)=\bA_{n}(u)=1$.

\begin{lem}\label{A-generators}The generators $\bA_k(u)$ of $\DO\YO$ satisfy the relations
\begin{equation}\label{a-rel}
\bA_k(u)\bx_l^+(v)\frac{2u-2v+\delta_{kl}}{2u-2v- \delta_{kl}}=\bx_l^+(v)\bA_k(u).
\end{equation}
\end{lem}

\begin{proof}
Consider the algebra $\DO\YO'$ generated by $\bA_k(u), \bx^+_k(u)$ with the defining relations~(\ref{a-rel}),~(\ref{14})~and~(\ref{15}). There is a homomorphism $\phi:\DO\YO\to\DO\YO'$ such that \begin{multline*}\phi(\bx^+_k(u))=\bx^+_k(u),\\ \phi(\bh_k(u))=\bA_k(u+\frac{1}{2})^{-1}\bA_k(u-\frac{1}{2})^{-1}\bA_{k-1}(u)\bA_{k+1}(u)(u+\frac{1}{2})^{d_k}(u-\frac{1}{2})^{d_k}u^{-d_{k-1}}u^{-d_{k+1}}
\end{multline*}
Let $\DO\YO^+$ be the algebra generated by $\bx_l^+(u)$ with the defining relations~~(\ref{14})~and~(\ref{15}). The quotient of $\BC[\bA_{k,r}]_{r=1}^{\infty}\cdot\DO\YO^+$ by the relation~(\ref{a-rel}) is $\BC[\bA_{k,r}]_{r=1}^{\infty}\otimes\DO\YO^+$ as a filtered vector space. Hence the $\DO\YO'=\BC[\bA_{k,r}]_{r=1}^{\infty}\otimes\DO\YO^+$ as a filtered vector space. One can inductively express $\bA_{k,r}$ via $\phi(\bh_{k,s})$ with $s\le r+1$, hence $\DO\YO'$ is generated by $\phi(\bh_k(u))$ and $\bx_l^+(u)$. Hence $\phi$ is an isomorphism.
\end{proof}

\begin{lem}\label{yangian-rel}
Let $\bA_k(u)$ and $\bx_l^+(u)$ be the generating series of $\DO\YO$. Then the series $$
\ba_k(u)=\frac{\bA_k(u-\frac{1}{2})}{\bA_k(u+\frac{1}{2})}=1-d_ku^{-1}-\sum\limits_{r=1}^\infty\ba_{k,r}u^{-r-1},\quad \bx_l^+(u)
$$ satisfies the following commutator relations
\begin{equation}\label{a-x-defnrel}
[\ba_k(u),\bx_l^+(v)](u-v)=-\frac{\delta_{kl}}{u-v}\bx_l^+(v)\ba_k(u),\quad [\ba_k(u),\ba_l(v)]=0.
\end{equation}
The series $\ba_k(u),\ \bx_l^+(u)$ generate $\DO\YO$ with the defining relations~(\ref{a-x-defnrel}),~(\ref{14})~and~(\ref{15}), and their negative Fourier components generate $\YO$.
\end{lem}

\begin{proof} For $k\ne l$ the relation is obvious, for $k=l$ we have
$$
\ba_k(u)\bx_k^+(v)\frac{u-\frac{1}{2}-v+\frac{1}{2}}{u-\frac{1}{2}-v-\frac{1}{2}}\cdot\frac{u+\frac{1}{2}-v-\frac{1}{2}}{u+\frac{1}{2}-v+\frac{1}{2}}=\bx_k^+(v)\ba_k(u).
$$therefore
$$
\ba_k(u)\bx_k^+(v)\frac{(u-v)^2}{(u-v)^2-1}=\bx_k^+(v)\ba_k(u).
$$

One can inductively express $\bA_{k,r}$ via $\ba_{k,s}$ with $s\le r+1$, hence $\DO\YO$ is generated by $\ba_k(u)$ and $\bx_l^+(u)$. On the other hand, the quotient of $\BC[\ba_{k,r}]_{r=1}^{\infty}\cdot\DO\YO^+$ by the relation~(\ref{a-x-defnrel}) is $\BC[\ba_{k,r}]_{r=1}^{\infty}\otimes\DO\YO^+$ as a filtered vector space. The same argumentation for $\YO$. Hence the assertion.
\end{proof}


\subsection{Classical limit of the Yangian} Consider the filtration on the Yangian $Y(\fsl_n)$ from \cite{mbook}, section 1.4, such that the associated graded algebra $\gr Y(\fsl_n)$ is commutative. Then $\gr Y(\fsl_n)=S(\fsl_n[t])$ as a graded commutative algebra ($\deg x\otimes t^r=r+1$ for $x\in\fsl_n$). The Poisson bracket on $\gr Y(\fsl_n)$ has the degree $-1$ and deforms the Lie-Poisson bracket on $S(\fsl_n[t])$. We have $\deg\bA_{i,r}=\deg\bx_{k,r}^+=r+1$ with respect to this filtration. Due to the PBW property of the Yangian, the subalgebra $\gr\YO\subset\gr Y(\fsl_n)$ is generated as a Poisson algebra by (the leading terms of) $\bA_{i,r},\ \bx_{k,r}^+$, subject to the relations:
\begin{equation}
\{\bA_k(u),\bA_l(v)\}=0;
\end{equation}
\begin{equation}
\{\bA_k(u),\bx_l^+(v)\}(u-v)=-\delta_{kl}\bA(u)\bx^+(v);
\end{equation}
\begin{equation}
\{\bx_k^+(u),\bx_l^+(v)\}(u-v)=c_{kl}\bx^+(u)\bx^+(v);
\end{equation}
in the sense that negative Fourier components of LHS and RHS are equal.
\begin{equation}
\{\bx_{k,r}^+,\{\bx_{k,p}^+,\bx_{l,s}^+\}\}+
\{\bx_{k,p}^+,\{\bx_{k,r}^+,\bx_{l,s}^+\}\}=0,\
k=l\pm1,\ \forall p,r,s\in\BN.
\end{equation}

\subsection{Affine Yangian}

As for the finite case, we will consider the ``affine Borel Yangian''. This is an associative algebra $\widehat{\YO}_{\beta}$ depending on $\beta\in\BC$, generated by the series
\begin{equation}
\bx_k^+(u):=1+\sum\limits_{r=0}^{\infty}\bx_{k,r}u^{-r-1},
\end{equation}
\begin{equation}
\bA_k(u):=u^{d_k}+\sum\limits_{r=0}^{\infty}\bA_{k,r}u^{d_k-r-1},
\end{equation}
with $k\in\BZ$ subject to the relations

\begin{equation}
\bA_{k+n}(u)=\bA_k(u+\beta),\quad \bx^+_{k+n}(u)=\bx^+_k(u+\beta);
\end{equation}
\begin{equation}
\bx_k^\pm(u)\bx_l^\pm(v)(2u-2v\mp c_{kl})=\bx_l^\pm(v)\bx_k^\pm(u)(2u-2v\pm c_{kl}),
\end{equation}
where $(c_{kl})$ stands for the Cartan matrix of $A_\infty$;
\begin{equation}
\bA_k(u)\bx_l^+(v)\frac{2u-2v+\delta_{kl}}{2u-2v- \delta_{kl}}=\bx_l^+(v)\bA_k(u).
\end{equation}
in the sense that negative Fourier components of LHS and RHS are equal.
\begin{equation}
[\bx_{k,r}^\pm,[\bx_{k,p}^\pm,\bx_{l,s}^\pm]]+
[\bx_{k,p}^\pm,[\bx_{k,r}^\pm,\bx_{l,s}^\pm]]=0,\
k=l\pm1,\ \forall p,r,s\in\BN.
\end{equation}

\begin{rem} \emph{The algebra $\widehat{\YO}_{\beta}$ does not depend on $\ul{d}$: one can multiply the generating series $\bA_k(u)$ by any Laurent series $F_k(u^{-1})$ with constant coefficients (with the only condition $F_{k+n}(u)=F_k(u+\beta)$), and the relations remain the same.}
\end{rem}

\begin{rem} \emph{The algebra $\widehat{\YO}_{\beta}$ } is not \emph{a Borel subalgebra of the affine Yangian $\widehat Y_{\beta}$ of type $\widehat{A}_{n-1}$ from~\cite{g}. It is related to the Langlands dual of the $\widehat{\fsl_n}$.}
\end{rem}

Lemma~\ref{yangian-rel} and Lemma~\ref{A-generators} are also true for the affine Yangian, and the proof is the same.

\subsection{Yangian and finite Zastava spaces. $\fsl_2$ case.}

\begin{prop}\label{yangian-qutient-sl_2}
The algebra $\YO_{d}$ is a quotient of the Borel Yangian $\YO$ of $\fsl_2$ by the relations $\bA_r=0\ \text{for}\ r>d$.
\end{prop}

\begin{proof}
Consider the following generating series of the quantized coordinate ring $\YO_d$ of the Zastava space
$$
b(u)=\sum\limits_{r=0}^\infty b_ru^{-r-1}\ \text{and}\ a(u)=1-du^{-1}-\sum\limits_{r=1}^\infty a_ru^{-r-1}.
$$
According to Proposition~\ref{prop-quantum-relations} we have
\begin{equation*}[b(u),b(v)](u-v)=b(u)b(v)+b(v)b(u),
\end{equation*}
\begin{equation*}
[a(u),b(v)](u-v)=-\frac{1}{u-v}b(v)a(u).
\end{equation*}
in the sense that negative Fourier components of LHS and RHS are equal.

Hence there is an epimorphism $\varphi_d:\YO\to\YO_d$ sending $\ba(u)$ to $a(u)$ and $\bx^+(u)$ to $b(u)$. Thus $\YO_d$ is a quotient of $\YO$.

Let
\begin{equation*}
D(u)=u^{d}+D_0u^{d-1}+D_1u^{d-2}+\ldots+D_{d-1}
\end{equation*}
be the Capelli determinant of the matrix $(e_{ij})$ (see \cite{mbook}, (7.5)). According to the Newton identity (see Theorem~7.1.3 of \cite{mbook}), we have
\begin{equation*}
a(u)=\frac{D(-u+d)}{D(-u+d-1)}.
\end{equation*}
This means that $\varphi_d(\bA(u))=D(-u+d-\frac{1}{2})$. In particular, $\varphi_d(\bA_r)=0$ for $r>d$.

To prove that $\YO_{d}$ is a quotient of $\YO$ by the relations $\bA_r=0\ \text{for}\ r>d$ it suffices to show this in quasiclassical limit. Namely, we have to show that the coordinate ring $\BC[\fZ_d]$ is a quotient of the classical limit of the Yangian by the Poisson ideal generated by $\bA_r\ \text{for}\ r>d$. The ring $\BC[\fZ_d]$ is generated by the coefficients of the characteristic polynomial of the matrix $A$, which are the $D_r$'s, and by the $b_r$'s with the defining relations $$D_r=0\ \text{for}\ r>d$$ and $$b_s+\sum\limits_{r=0}^{d-1}b_{s-r-1}D_r=0\ \text{for}\ s\ge d.$$ We have $\{\bA_r,\bx^+_{s-d}\}=-\bx^+_s+\sum\limits_{r=0}^{d-1}(-1)^r\bx^+_{s-r-1}\bA_r$, hence the second relation also belongs to the Poisson ideal generated by $\bA_r\ \text{for}\ r>d$. Hence the assertion.
\end{proof}

\subsection{Yangian and finite Zastava spaces. General case.}

The quantized coordinate ring of the Zastava space $\YO_{\ul{d}}$ is generated by the coefficients of $a_{k}(u)$ and $b_l(u)$ (or $b'_l(u)=b_l(u+d_l)$) for $0<k,l<n$. According to Proposition~\ref{prop-quantum-relations-general} we have
\begin{equation*} 2(u-v)[b_k(u),b'_l(v)]=c_{kl}(b_k(u)b'_l(v)+b'_l(v)b_k(u)),
\end{equation*}

\begin{equation*}
[a_k(u),b_l(v)](u-v)=-\frac{\delta_{kl}}{u-v}b_l(v)a_k(u).
\end{equation*}
in the sense that negative Fourier components of LHS and RHS are equal.

\begin{prop}\label{yangian-quotient-general}
The algebra $\YO_{\ul{d}}$ is a quotient of the Borel Yangian $\YO$ of $\fsl_n$ by some ideal containing $\bA_{k,r}=0\ \text{for}\ r>d_k$.
\end{prop}
\begin{proof}
There is a homomorphism $\varphi_{\ul{d}}:\YO\to\YO_{\ul{d}}$ sending $\ba_k(u-\sum\limits_{m=1}^kd_m)$ to $a_k(u)$ and $x_l^+(u-\sum\limits_{m=1}^ld_m)$ to $b_l(u)$. The rest of the proof is the same as for the $\fsl_2$ case.
\end{proof}

\begin{conj} $\YO_{\ul{d}}=\YO/\{\bA_{k,r}\ |\ r>d_k\}$.
\end{conj}

\subsection{Affine Yangian and (deformed) affine Zastava}

\begin{prop}\label{yangian-quotient-affine}
For $|\ul{\mu}|\ne0$, the algebra $\YO_{\ul{d}}^{\ul{\mu}}$ is a quotient of $\widehat{\YO}_\beta$ (with $\beta=\sum\limits_{l=1}^n(d_l+\mu_l)$) by some ideal containing $\bA_{k,r}=0\ \text{for}\ r>d_k$.
\end{prop}
\begin{proof}
The same as for Proposition~\ref{yangian-quotient-general}.
\end{proof}

\begin{conj} For $|\ul{\mu}|\ne0$, we have $\YO_{\ul{d}}^{\ul{\mu}}=\widehat{\YO}_\beta/\{\bA_{k,r}\ |\ r>d_k\}$.
\end{conj}

\section{Speculations on quantization of Laumon spaces in finite characteristic}

For finite Laumon spaces $\CP_{\ul{d}},\ d_0=0$, a conjecture about quantum
cohomology was formulated in~\cite{fr} (Conjecture~4.8 and Corollary~4.10),
and proved recently by A.~Negut. It follows that the monodromy of the quantum
connection gives rise to an action of the pure braid group on the cohomology
of $\CP_{\ul{d}}$. According to Bridgeland-Bezrukavnikov-Okounkov philosophy,
if we transfer this action to the K-theory (via Chern character), then it should
come from an action of the pure braid group on the derived coherent category
of $\CP_{\ul{d}}$.
In this section we discuss various possibilities to construct an action of the
pure braid group on the equivariant derived category of coherent sheaves on
$\CP_{\ul{d}}$.

\subsection{Variation of stability conditions}
\label{vari}
We consider a vector $\zeta=(\zeta_l)_{l\in\BZ/n\BZ}$ with integral coordinates.
Let $\chi:=\prod_{l\in\BZ/n\BZ}\det_l^{-\zeta_l}$ stand for the corresponding
character of $G_{\ul{d}}$. Let $\BC[M^\Gamma_{\ul{d}}]^{\chi^r}\ (r\in\BN)$ stand
for the $\chi^r$-isotypic subspace of $\BC[M^\Gamma_{\ul{d}}]$.
Let $\CP^\zeta_{\ul{d}}$ stand for the projective spectrum of the graded algebra
$\bigoplus_{r\in\BN}\BC[M^\Gamma_{\ul{d}}]^{\chi^r}$. In particular,
$\CP^{(-1,\ldots,-1)}_{\ul{d}}=\CP_{\ul{d}}$. According to the GIT,
$\CP^\zeta_{\ul{d}}$ is the moduli space of $S$-equivalence classes of
$\zeta$-semistable $Q$-modules. We recall the required notions
following~\cite{na}.

A module $(V_\bullet,A_\bullet,B_\bullet,p_\bullet,q_\bullet)$ over the the
chainsaw quiver is called $\zeta$-semistable if\\
(a) for subspaces $V'_\bullet\subset V_\bullet$ such that
$B_l(V'_l)\subset V'_{l+1}$
and $A_l(V'_l)\subset V'_l$, and $\on{Ker}q_l\supset V'_l$ we have
$\langle\zeta,\ul{\dim}V'_\bullet\rangle\leq0$. Here $\langle\cdot,\cdot\rangle$
stands for the standard scalar product, i.e. the sum of products of
coordinates.\\
(b) for subspaces $V'_\bullet\subset V_\bullet$ such that
$B_l(V'_l)\subset V'_{l+1}$
and $A_l(V'_l)\subset V'_l$, and $\on{Im}p_l\subset V'_l$ we have
$\langle\zeta,\ul{\on{codim}}V'_\bullet\rangle\geq0$.

We say that a module $(V_\bullet,A_\bullet,B_\bullet,p_\bullet,q_\bullet)$ is
$\zeta$-stable if the above inequalities are strict unless
$V'_\bullet=0$ in (a), and $V'_\bullet=V_\bullet$ in (b). If $\zeta$ and
$\zeta'$ are proportional, the stability conditions coincide, so in the
definition of stability we may take the vectors $\zeta$ with rational
coordinates.

Let us reformulate these conditions in a slightly different way.
We set $\zeta_\infty:=-\langle\zeta,\ul{d}\rangle$, and $\widetilde{\zeta}:=
(\zeta,\zeta_\infty)$ (an $n+1$-dimensional vector). Recall that the line
$W_\infty$ is also a part of data of a $Q$-module, and now we allow to vary
the dimension of $W_\infty$ (in particular, we allow $W_\infty=0$), i.e. we
consider the abelian category of $Q$-modules. Accordingly, we introduce the
enhanced dimension $\widetilde{\ul{d}}=(\ul{d},d_\infty):=
(\ul{d},\dim W_\infty)$.
In case $W_\infty=0$ we assume $\langle\zeta,\ul{d}\rangle=0$.
Given a $Q$-module $Y$ with $\dim W_\infty\leq1$, and a $Q$-submodule
$Y'\subset Y$ of enhanced dimension $\widetilde{\ul{d}}{}'$ (where the last
coordinate is either 0 or 1) we define the slope by
$\theta_\zeta(Y'):=\frac{\langle\widetilde{\zeta},\widetilde{\ul{d}}{}'\rangle}
{\langle(1,\ldots,1),\widetilde{\ul{d}}{}'\rangle}$.
We say that a $Q$-module $Y$ is $\zeta$-semistable if for any nonzero
$Q$-submodule $Y'\subset Y$ we have $\theta_\zeta(Y')\leq\theta_\zeta(Y)$.
We say that $Y$ is $\zeta$-stable, if the above inequality is strict unless
$Y'=Y$. Note that for $\dim W_\infty=1$ the definition of the present paragraph
is equivalent to the definition of the previous paragraph.

Finally, we say that two $Q$-modules are $S$-equivalent, if their
Jordan-H\"older filtrations have the same composition factors.

\subsection{Walls}
\label{walls}
Given $l,l'\in\BZ/n\BZ$, let $[l,l']\subset\BZ/n\BZ$ stand for the interval
between $l$ and $l'$ ($l,l'$ included) in the natural cyclic order. Also,
given $0<l\leq l'<n$, let $[l,l']$
stand for the set $\{l,l+1,\ldots,l'-1,l'\}$. We say that a hyperplane
$H_{l,l'}:=\{\zeta:\ \sum_{k\in[l,l']}\zeta_k=0\}\subset\BQ^{\BZ/n\BZ}$ is an
{\em affine wall}; and also a hyperplane
$H:=\{\zeta:\ \sum_{\BZ/n\BZ}\zeta_l=0\}\subset\BQ^{\BZ/n\BZ}$ is an affine wall.
In case $d_0=0$, the coordinate $\zeta_0$ is irrelevant, and
the space of stability conditions is just $\BQ^{n-1}$. Given $0<l\leq l'<n$,
we say that a hyperplane
$H_{l,l'}:=\{\zeta:\ \sum_{k\in[l,l']}\zeta_k=0\}\subset\BQ^{n-1}$ is a
{\em finite wall}.

\begin{prop}
\label{semist}
(a) If $\zeta\in\BQ^{\BZ/n\BZ}$ does not lie on an affine wall,
$\zeta$-stability is equivalent to $\zeta$-semistability;

(b) If $d_0=0$, and $\zeta\in\BQ^{n-1}$ does not lie on a finite wall,
$\zeta$-stability is equivalent to $\zeta$-semistability.
\end{prop}

\proof
Given a $\zeta$-semistable $Q$-module $Y$ we consider its Harder-Narasimhan
filtration with $\zeta$-stable factors. If $Y$ is not $\zeta$-stable, there
are at least 2 factors, and at least one of them has $d_\infty=0$. Thus it
suffices to check that when $\zeta$ does not lie on a wall, then there are no
$\zeta$-stable $Q$-modules with $d_\infty=0$, i.e. $W_\infty=0$.

So we suppose $Y$ is a $\zeta$-stable $Q$-module with $W_\infty=0$.
By a standard argument, $Y$ does not have nonscalar endomorphisms. However,
the collection $(A_\bullet)$ is an endomorphism of $Y$. In effect, since
$W_\infty=0$, we have $p_\bullet=0=q_\bullet$, and hence the relation in $Q$
reads $A_{l+1}B_l-B_lA_l=0$, i.e. the operators $B_\bullet$ intertwine the
endomorphisms $A_\bullet$. We conclude that $A_l=c\on{Id}_{V_l}$ for some
constant $c$. Subtracting $c\on{Id}_{V_l}$ we may and will assume $A_l\equiv0$,
and thus we deal just with a representation of the cyclic quiver.
Moreover, the collection $(B_{l-1}B_{l-2}\ldots B_{l+1}B_l)$ is an endomorphism
of $Y$, and hence $B_{l-1}B_{l-2}\ldots B_{l+1}B_l\equiv c\on{Id}_{V_l}$.
In case $c\ne0$, we get $Y\simeq L(0,c)^{\oplus d}$ in notations of~\ref{strat}.
Being indecomposable, $Y$ is isomorphic to $L(0,c)$.
In case $c=0$, according to the well known classification of nilpotent
representations of a cyclic quiver, an indecomposable $Y$ must be of the form
$Y_{[l,l']}$. Here $Y_{[l,l']}$ has $V_k=\BC$ for $k\in[l,l']$, and $V_k=0$
otherwise; furthermore, $B_k$ is an isomorphism for $l'\ne k\in[l,l']$, and
$B_{l'}=0$. Finally, if $d_0=0$, only $Y_{[l,l']}$ with $0<l\leq l'<n$ occur.

It remains to classify the stability conditions $\zeta$ for which
$L(0,c)$ or $Y_{[l,l']}$ are stable. For irreducible $L(0,c)$ any
$\zeta$ on the wall
$H:=\{\zeta:\ \sum_{\BZ/n\BZ}\zeta_l=0\}\subset\BQ^{\BZ/n\BZ}$ works, and no
other $\zeta$ works.
For $Y_{[l,l']}$ any submodule is of the form $Y_{[l'',l']}$ for
$l\leq l''\leq l'$ in the cyclic order. It follows that $Y_{[l,l']}$ is
$\zeta$-stable iff $\zeta_l+\zeta_{l+1}+\ldots+\zeta_{l'-1}+\zeta_{l'}=0$,
and $\zeta_l+\zeta_{l+1}+\ldots+\zeta_{l''}\geq0$ for any  $l\leq l''\leq l'$
in the cyclic order. This completes the proof of the proposition. \qed

\subsection{Smoothness}
\label{smo}
For a $\zeta$-stable $Q$-module $Y$ the stabilizer of $Y$ in $G_{\ul{d}}$ is
trivial by the standard argument we have used already: the stable
modules do not admit nonscalar endomorphisms.
Recall that we have a morphism $\mu:\ M^\Gamma_{\ul{d}}\to\bigoplus_{l\in\BZ/n\BZ}
\Hom(V_l,V_{l+1}),\ (A_\bullet,B_\bullet,p_\bullet,q_\bullet)\mapsto
(A_{l+1}B_l-B_lA_l+p_{l+1}q_l)_{l\in\BZ/n\BZ}$, and $\sM_{\ul{d}}=\mu^{-1}(0)$.
In the theory of Nakajima quiver varieties, the moduli space of $\zeta$-stable
quiver representations is smooth because the differential of the moment map
is surjective. In our situation this is no longer true as the following example
shows.

We identify the tangent space to the vector space $M^\Gamma_{\ul{d}}$ at
$Y=(A_\bullet,B_\bullet,p_\bullet,q_\bullet)$ with
$M^\Gamma_{\ul{d}}$, we also identify the tangent space to the vector space
$\bigoplus_{l\in\BZ/n\BZ}\Hom(V_l,V_{l+1})$ at $\mu(Y)$ with this vector space,
and write down the formula for the differential $d\mu(Y)$ as follows:
$d\mu(A'_\bullet,B'_\bullet,p'_\bullet,q'_\bullet)=
(A'_{l+1}B_l+A_{l+1}B'_l-B_lA'_l-B'_lA_l+p'_{l+1}q_l+p_{l+1}q'_l)_{l\in\BZ/n\BZ}$.
The differential $d\mu(Y)$ is not surjective iff there exists nonzero
$(C_l\in\Hom(V_l,V_{l-1}))_{l\in\BZ/n\BZ}$ orthogonal to the image of
$d\mu$ with respect to the pairing given by the trace of the product.
Equivalently, $C_lA_l-A_{l-1}C_l=0,\ B_{l-1}C_l=C_{l+1}B_l=0,\ q_{l-1}C_l=0,\
C_lp_l=0$ for any $l$.

Now let us recall the setup of Example~\ref{sl3} and take the stability
condition $\zeta=(\zeta_1,\zeta_2)=(-1,2)$ lying off the walls.
We take $A_1=A_2=p_1=q_2=1,\ B_1=p_2=q_1=0$. It is immediate to check that
$Y$ is $\zeta$-stable but on the other hand $C_2=1$ satisfies the above
conditions. Hence the moduli space $\CP_{1,1}^{(-1,2)}$ of $\zeta$-stable
(equivalently, $\zeta$-semistable) $Q$-modules is {\em nonsmooth}.
In fact, it is easy to check that $\CP_{1,1}^{(-1,2)}\iso\fZ_{1,1}$.


\subsection{Localization in characteristic $p$}
\label{loc}
From now on we assume that the base field is $\sK:=\overline{\mathbb F}_p$,
an algebraic closure of a finite field of characteristic $p\gg0$.
We will use the notations and results of section~3 of~\cite{bfgi}.
For any algebraic variety $X$ over $\sK$ we denote by $X^{(1)}$ its Frobenius
twist, and we denote by $\on{Fr}:\ X\to X^{(1)}$ the Frobenius morphism.
For a connected linear algebraic group $A$ over $\sK$ we have an exact
sequence of groups $1\to A_1\to A\stackrel{\on{Fr}}{\longrightarrow}A^{(1)}\to1$
where $A_1$ stands for the Frobenius kernel. The Lie algebra $\fa$ of $A$
is equipped with a natural structure of $p$-Lie algebra, and its universal
enveloping algebra $U(\fa)$ contains the $p$-center $\fZ(\fa)$.
We denote by ${\mathbb X}^*(\fa)$ the lattice of characters of $\fa$ of the
type $\chi=\on{dlog}f$ where $f:\ A\to{\mathbb G}_m$ is an algebraic character
of $A$. For such a character $\chi$ we denote by $I_\chi$ the kernel of the
corresponding homomorphism $U(\fa)\to\sK$. We set
$I_\chi^{(1)}:=I_\chi\cap\fZ(\fa)$, a maximal ideal of $\fZ(\fa)$.
Note that $I_\chi^{(1)}=I_0^{(1)}$ (see~\cite{bfgi}~(3.2.4) and~3.3).
We denote by $\fu(\fa)$ the quotient of $U(\fa)$ by the two-sided ideal
generated by $I^{(1)}_\chi=I^{(1)}_0$. The image of $I_\chi$ in $\fu(\fa)$
is denoted by $\fri_\chi\subset\fu(\fa)$.

Now we take $A=G_{\ul{d}}=\prod_{l\in\BZ/n\BZ}GL(V_l)$, and we denote its
Lie algebra by $\fg_{\ul{d}}$. We have (notations of~(\ref{2sided-ideal})
and~(\ref{poisson-ideal}))
$R\cap\fZ(\fa)\simeq\sK[S_{\ul{d}}^{(1)}]$. Thus we may localize the
$(R\cap\fZ(\fa))$-module $R$ to $S_{\ul{d}}^{(1)}$ and view it as a sheaf
of algebras $\CR_{\ul{d}}^{(1)}$.
We have the moment map $\mu^{(1)}:\ S_{\ul{d}}^{(1)} \to[\fg_{\ul{d}}^*]^{(1)}$.
The quotient $\CR_{\ul{d},\chi}^{(1)}:=
\CR_{\ul{d}}^{(1)}/(\CR_{\ul{d}}^{(1)}\cdot I^{(1)}_\chi)$ is just
the restriction of $\CR_{\ul{d}}^{(1)}$ to the scheme-theoretic zero-fiber
of the moment map (and is independent of $\chi\in{\mathbb X}^*(\fg_{\ul{d}})$).
This zero-fiber is nothing else than $\sM_{\ul{d}}^{(1)}$.
We consider $\CE_\chi:=(\CR_{\ul{d},\chi}^{(1)}/\CR_{\ul{d},\chi}^{(1)}\cdot
\fri_\chi)^{G_{\ul{d},1}}$: a $G_{\ul{d}}^{(1)}$-equivariant sheaf on
$\sM_{\ul{d}}^{(1)}$.

We restrict $\CE_\chi$ to the open subset $\sM_{\ul{d}}^{(1),s}\subset
\sM_{\ul{d}}^{(1)}$ of stable points. The action of $G_{\ul{d}}^{(1)}$ on
$\sM_{\ul{d}}^{(1),s}$ is free; the projection $\on{pr}:\ \sM_{\ul{d}}^{(1),s}\to
\CP_{\ul{d}}^{(1)}$ is a $G_{\ul{d}}^{(1)}$-torsor. We set
\begin{equation}
\label{locp}
\CA_\chi:=
\on{pr}_*(\CE_\chi|_{\sM_{\ul{d}}^{(1),s}})^{G_{\ul{d}}^{(1)}},\ \on{and\ put}\
\sA_\chi:=\Gamma(\CP_{\ul{d}}^{(1)},\CA_\chi).
\end{equation}
Given another character $\psi\in{\mathbb X}^*(\fg_{\ul{d}})$ we
consider the $G_{\ul{d}}^{(1)}$-equivariant sheaf
$_\chi\CE_\psi:={\underline\Hom}{}_{\CR_{\ul{d},\chi}^{(1)}}(\CR_{\ul{d},\chi}^{(1)}/
\CR_{\ul{d},\chi}^{(1)}\cdot\fri_\chi,\CR_{\ul{d},\chi}^{(1)}/
\CR_{\ul{d},\chi}^{(1)}\cdot\fri_\psi)$ on $\sM_{\ul{d}}^{(1)}$. We set
\begin{equation}
\label{locp bimod}
_\chi\CA_\psi:=\on{pr}_*(\ _\chi\CE_\psi|_{\sM_{\ul{d}}^{(1),s}})^{(\psi-\chi)}
\end{equation}
where the superscipt $(\psi-\chi)$ stands for the $(\psi-\chi)$-weight
component. This is an $\CA_\chi-\CA_\psi$-bimodule.

\begin{conj}
\label{localiz}
(a) The canonical algebra morphism $\Xi_\chi:\ \CY_{\ul{d}}\to\sA_\chi
=\Gamma(\CP_{\ul{d}}^{(1)},\CA_\chi)$ is an algebra isomorphism.

(b) The algebra $\CY_{\ul{d}}$ has finite homological dimension.

(c) The functor of global sections $R\Gamma(\CP_{\ul{d}}^{(1)},?)$
from the bounded derived category $D^b(\CA_\chi-\on{Mod})$ of $\CA_\chi$-modules
to the bounded derived category $D^b(\sA_\chi-\on{Mod})$ of
$\sA_\chi$-modules is an equivalence of categories for $\chi=0$.

(d) The bimodules $_\chi\CE_\psi$ give rise to the Morita equivalences
$_\chi E_\psi:\ \CA_\chi-\on{Mod}\iso\CA_\psi-\on{Mod}$.
\end{conj}

Let us say that $\chi\in{\mathbb X}^*(\fg_{\ul{d}})$ is {\em regular} if
the functor $R\Gamma(\CP_{\ul{d}}^{(1)},?):\ D^b(\CA_\chi-\on{Mod})\to
D^b(\sA_\chi-\on{Mod})$ is an equivalence of categories. Thus for regular
$\chi$ we get an equivalence of categories $D^b(\CA_\chi-\on{Mod})\iso
D^b(\CY_{\ul{d}}-\on{Mod})$. Composing it with the Morita equivalences
$_\chi E_\psi$ for other regular characters $\psi$, we obtain the
self-equivalences $_\chi\varepsilon_\psi:\ D^b(\CY_{\ul{d}}-\on{Mod})\iso
D^b(\CY_{\ul{d}}-\on{Mod})$. We conjecture that they generate an action
of the pure (affine) braid group on $D^b(\CY_{\ul{d}}-\on{Mod})$.

\subsection{Splitting module}
\label{splitt}
Let us denote by $\CY_{\ul{d}}^0$ the ``Cartan'' subalgebra of $\CY_{\ul{d}}$
generated by $\{a_{l,r},\ l\in\BZ/n\BZ,\ r\geq0\}$, and let us denote
by $\CY_{\ul{d}}^+$ the ``nilpotent'' subalgebra of $\CY_{\ul{d}}$ generated by
$\{b_{l,r},\ l\in\BZ/n\BZ,\ r\geq0\}$.
We define the {\em $p$-center} $\fZ(\CY_{\ul{d}})$ as the Hamiltonian reduction of the $p$-center of $U(\fa_{\ul{d}})$ inside
$\CY_{\ul{d}}$. We have $\fZ(\CY_{\ul{d}})\simeq\sK[\fZ_{\ul{d}}^{(1)}]$ (it is just the coordinate ring of the Frobenius twist of the classical Hamiltonian reduction $\fZ_{\ul{d}}$).
There are also the ``Cartan'' subalgebra $\fZ(\CY_{\ul{d}})^0:=\fZ(\CY_{\ul{d}})\cap\CY_{\ul{d}}^0$, and
the ``nilpotent'' subalgebra $\fZ(\CY_{\ul{d}})^+:=\fZ(\CY_{\ul{d}})\cap\CY_{\ul{d}}^+$ inside this $p$-center.
Clearly, $\fZ(\CY_{\ul{d}})^0\simeq\sK[\BA^{\ul{d},(1)}]$ (regular functions
on the Frobenius twist of $\BA^{\ul{d}}$). We denote by
${\widehat\fZ}(\CY_{\ul{d}})^0$ the completion of
$\fZ(\CY_{\ul{d}})^0\simeq\sK[\BA^{\ul{d},(1)}]$ at the maximal ideal of the
point $0\in\BA^{\ul{d},(1)}$. We set


\begin{equation}
\label{splitM}
{\widehat M}:={\widehat\fZ}(\CY_{\ul{d}})^0\otimes_{\fZ(\CY_{\ul{d}})^0}\fZ(\CY_{\ul{d}})\otimes_{\fZ(\CY_{\ul{d}})^+}\CY_{\ul{d}}^+
\end{equation}
and we conjecture that the regular action of the algebra
${\widehat\fZ}(\CY_{\ul{d}})^0\otimes\CY_{\ul{d}}^+$ on $M$ extends to the
action of ${\widehat\CY}_{\ul{d}}$: the completion of $\CY_{\ul{d}}$ at the
maximal ideal of $\fZ(\CY_{\ul{d}})^0$. Moreover, we conjecture that the
action of $\bT^{(1)}=T^{(1)}\times{\mathbb G}_m^{(1)}\times{\mathbb G}_m^{(1)}$
on $\fZ(\CY_{\ul{d}})\simeq\sK[\fZ_{\ul{d}}^{(1)}]$ extends to an action of
$\bT^{(1)}$ (i.e. a grading) on $\widehat M$, and on ${\widehat\CY}_{\ul{d}}$.

\subsection{Coherent sheaves}
\label{cohs}
Let $\chi\in{\mathbb X}^*(\fg_{\ul{d}})$ be a regular character. We conjecture
that the equivalence $R\Gamma(\CP_{\ul{d}}^{(1)},?):\ D^b(\CA_\chi-\on{Mod})\iso
D^b(\CY_{\ul{d}}-\on{Mod})$ of Conjecture~\ref{localiz} extends to
$D^b_{\bT^{(1)}}({\widehat\CA}_\chi-\on{Mod})\iso
D^b_{\bT^{(1)}}({\widehat\CY}_{\ul{d}}-\on{Mod})$ where
${\widehat\CA}_\chi-\on{Mod}$ stands for the category of $\CA_\chi$-modules
supported set-theoretically over $0\in\BA^{\ul{d},(1)}$. Thus we obtain
the self-equivalences $_\chi\varepsilon_\psi:\
D^b_{\bT^{(1)}}({\widehat\CY}_{\ul{d}}-\on{Mod})\iso
D^b_{\bT^{(1)}}({\widehat\CY}_{\ul{d}}-\on{Mod})$. Let us denote by
${\widehat\CM}_\chi$ an ${\widehat\CA}_\chi$-module such that
$R\Gamma(\CP_{\ul{d}}^{(1)},{\widehat\CM}_\chi)=\widehat M$ (the localization
of $\widehat M$). We have a tensor product functor
\begin{equation}
\label{tensorM}
\tau:\ \on{Coh}_{\bT^{(1)}}({\widehat\CP}{}_{\ul{d}}^{(1)})\to
\on{Coh}_{\bT^{(1)}}({\widehat\CA}_\chi),\ \CF\mapsto{\widehat\CM}_\chi
\otimes_{\CO_{\CP_{\ul{d}}^{(1)}}}\CF
\end{equation}
where $\on{Coh}_{\bT^{(1)}}({\widehat\CP}{}_{\ul{d}}^{(1)})$ stands for the
category of coherent sheaves on ${\widehat\CP}{}_{\ul{d}}^{(1)}$ supported
set-theoretically over $0\in\BA^{\ul{d},(1)}$. We conjecture that $\tau$
is a full embedding onto the minimal Serre subcategory containing
${\widehat\CM}_\chi$. Moreover, the composition
$R\Gamma(\CP_{\ul{d}}^{(1)},?)\circ\tau$ is a full embedding
$\Upsilon:\ D^b\on{Coh}_{\bT^{(1)}}({\widehat\CP}{}_{\ul{d}}^{(1)})\to
D^b_{\bT^{(1)}}({\widehat\CY}_{\ul{d}}-\on{Mod})$. Finally, we expect that
the essential image of $\Upsilon$ is independent of the regular character
$\chi$ and is invariant under the equivalences $_\chi\varepsilon_\psi$.
All in all we obtain the desired action of the pure braid group on
$D^b\on{Coh}_{\bT^{(1)}}({\widehat\CP}{}_{\ul{d}}^{(1)})$ generated by the
equivalences $_\chi\varepsilon_\psi$.

\end{document}